\documentclass{amsart}
\usepackage[utf8]{inputenc}
\usepackage[T1]{fontenc}

\usepackage{hyperref}
\usepackage{graphicx}

\usepackage{constants}

\usepackage{amssymb, amsmath, amsthm}
\usepackage{MnSymbol}

\newtheorem{defi}{Definition}
\newtheorem{theorem}[defi]{Theorem}
\newtheorem{remark}[defi]{Remark}
 \newtheorem{prop}[defi]{Proposition}
\newtheorem{lemma}[defi]{Lemma}
\newtheorem{cor}[defi]{Corollary}

\usepackage{color}

\newcommand{\R}{\mathbb R}
\newcommand{\N}{\mathbb N}
\newcommand{\Z}{\mathbb Z}

\title[The Steklov spectrum and coarse discretizations]{The Steklov spectrum and coarse discretizations of manifolds with boundary}
\author{Bruno Colbois}
\address{Universit\'e de Neuch\^atel, Institut de Math\'ematiques, Rue
  Emile-Argand 11, CH-2000 Neuch\^atel, Switzerland}
\email{bruno.colbois@unine.ch}
\author{Alexandre Girouard}
\address{D\'e\-par\-te\-ment de math\'ematiques et de
sta\-tistique, Universit\'e Laval, Pavillon Alexandre-Vachon,
1045, av. de la M\'edecine,
Qu\'ebec Qc G1V 0A6, 
Canada }
\email{alexandre.girouard@mat.ulaval.ca}
\author{Binoy Raveendran}
\address{ Indian Institute of Technology, Kanpur, India.}
\email{binoy1729@gmail.com}

\date{\today}

\begin{document}

\begin{abstract}
Given $\kappa, r_0>0$ and $n\in\N$, we consider the class
$\mathcal{M}=\mathcal{M}(\kappa,r_0,n)$ of compact $n$-dimensional
Riemannian manifolds with cylindrical boundary, Ricci curvature
bounded below by $-(n-1)\kappa$ and injectivity radius bounded below
by $r_0$ away from the boundary. For a manifold $M\in\mathcal{M}$ we introduce a notion of
discretization, leading to a graph with boundary which is roughly
isometric to $M$, with constants depending only on $\kappa,r_0,n$. In
this context, we prove a uniform spectral comparison inequality
between the Steklov eigenvalues of a manifold $M\in\mathcal{M}$ and
those of its discretization. Some applications to the construction of
sequences of surfaces with boundary of fixed length and with
arbitraritly large
Steklov spectral gap $\sigma_2-\sigma_1$ are given. In particular, we obtain such a
sequence for surfaces with connected boundary. The applications are
based on the construction of graph-like surfaces which are obtained from
sequences of graphs with good expansion properties. 
\end{abstract}

\maketitle


\section{Introduction}
\label{section:introduction}
Let $M$ be a smooth compact Riemannian manifold of dimension $n\geq 2$, with
smooth boundary $\Sigma=\partial M$. 
The
\emph{Steklov problem} on $M$ is to find all $\sigma\in\mathbb{R}$ for
which there exists a non-zero function $u$ such that
\begin{equation*}
\label{stek}
\begin{cases}
  \Delta u=0& \mbox{ in } M,\\
  \partial_\nu u=\sigma \,u& \mbox{ on }\Sigma,
\end{cases}
\end{equation*}
where $\Delta$ is the Laplace-Beltrami operator acting on functions on
$M$, and $\partial_\nu$ is the outward normal derivative along the
boundary $\Sigma$. 
It is standard that the Steklov spectrum is discrete:
$$0=\sigma_1(M)\leq\sigma_2(M)\leq\cdots\nearrow\infty.$$

Our goal in this paper is to understand the interplay between the Steklov
eigenvalues $\sigma_j(M)$ and large scale geometric features of the space $M$. 
Discretization methods which are classical in the context of the Laplace operator are developed in the context of manifolds with boundary. 
In order to single out large scale
phenomena, we restrict our attention to a class of manifolds with
bounded geometry. Throughout the paper, we therefore assume the
existence of constants $\kappa\geq0$ and $r_0\in (0,1)$ such
that 
\begin{itemize}
\item[H1)] The boundary $\Sigma$ admits a neighbourhood
which is isometric to the cylinder $[0,1]\times\Sigma$, with the
boundary corresponding to $\{0\}\times\Sigma$;
\item[H2)] The Ricci curvature of $M$ is bounded below by $-(n-1)\kappa$;
\item[H3)] The Ricci curvature of $\Sigma$ is bounded below by $-(n-2)\kappa$;
\item[H4)] For each point $p\in M$ such that $d(p,\Sigma)>1$, $\mbox{inj}_M(p)>r_0;$
\item[H5)] For each point $p\in \Sigma$, $\mbox{inj}_\Sigma(p)>r_0.$
\end{itemize}
The class of compact $n$-dimensional Riemannian manifolds with smooth boundary satisfying these hypotheses
is denoted $\mathcal{M}=\mathcal{M}(\kappa, r_0,n).$ 

\begin{remark}\label{remark:naturalC}
  The conditions defining the class $\mathcal{M}$ are natural. Indeed we
  use methods from \cite{cha2} on the boundary $\Sigma$ of manifolds
  $M\in\mathcal{M}$ which require curvature and injectivity lower
  bounds. The product structure near the boundary is used to avoid
  situations where the Steklov eigenvalues become very small and could
  not be detected through a discretization. For instance, it is known
  that if two boundary components become close to each others, then
  each $\sigma_j$ tends to zero. This could happen for a
  very short cylinder (see Lemma 6.1 of \cite{ceg2}) or for a surface
  which has a "thin passage" as described by Figure 3 of \cite[Section
    4]{gpsurvey}).  
    It is also of interest to study the effect of drastic perturbations of the the Riemannian metric away from the boundary. This is the subject of a forthcoming paper \cite{ceg3}. The above conditions ensure that none of these situations will occur for manifolds in $\mathcal{M}$.
\end{remark}


\subsection{Discretization and spectral comparison}
Our aim is to study spectral properties of manifolds in the
class $\mathcal{M}$ up to rough isometries.
\begin{defi}\label{definition:roughisometry}
  A \emph{rough isometry} between two metric spaces $X$ and $Y$ is a
  map $\Phi:X\rightarrow  Y$ such that, there exist constants $a\geq
  1, b\geq 0, \tau \geq 0$ satisfying 
  \begin{gather}\label{equation:roughisometry}
    a^{-1}d(x_1,x_2) - b \leq d(\Phi(x_1), \Phi(x_2)) \leq a\,d(x_1,x_2) + b
  \end{gather}
  for every $x_1, x_2 \in X$ and which satisfies 
  $$\bigcup_{x\in X} B(\Phi(x), \tau) = Y.$$
\end{defi}
Given $\epsilon\in (0,r_0/4)$, we will define a discretization procedure
\begin{center}
  $\mathcal{M}(\kappa, r_0,n)\xrightarrow{\hspace{2cm}}$ Graphs with boundary
\end{center}
such that the $\epsilon$-discretization $\Gamma_M$ of a manifold $M$ is roughly
isometric to $M$ with constants $a,b,\tau$ controlled in terms of the
geometric constraints defining the class $\mathcal{M}(\kappa, r_0,n)$.
A graph with boundary is simply a graph $\Gamma=(V,E)$ with a
distinguished set of vertices $B\subset V$ that is treated as a
boundary. In Section \ref{section:roughlyisometricgraphs} we will
introduce a natural notion of Steklov spectrum on graphs with
boundary:
$$0=\sigma_1(\Gamma,B)\leq\sigma_2(\Gamma,B)\leq\cdots\leq\sigma_k(\Gamma,B),$$
where $k=|B|$ is the number of vertices in the boundary.

Our main goal is to establish a relation between the Steklov eigenvalues of
$M$ and the Steklov eigenvalues of the discretization of $M$.
\begin{theorem}\label{theorem:MainSpectralComparison}
  Given $\epsilon\in(0,r_0/4)$, there exist numbers $a,b>0$ depending on $\kappa,r_0,n$ and $\epsilon$ such that
  any $\epsilon$-discretization $(\Gamma_M,V_\Sigma)$ of a manifold 
  $M\in\mathcal{M}(\kappa,r_0,n)$ satisfies	
  \begin{gather}\label{inequality:main1}
  a<\frac{\sigma_2(M)}{\sigma_2(\Gamma,V_\Sigma)}<b.
  \end{gather}	
Moreover, if $\sigma_k(M)$ is small enough, a similar result holds for
for each $k\le |V_\Sigma|$: there exists a constant $C>0$ (which
depends at most on $\kappa,\epsilon,n$) such that, if $\sigma_k(M) \le
C/k$, then 
\begin{gather}\label{inequality:main2} 
   a<\frac{\sigma_k(M)}{\sigma_k(\Gamma,V_\Sigma)}<b.
\end{gather}
Without this last hypothesis, the following weaker estimate holds for each $k\le |V_\Sigma|$:
\begin{gather}\label{inequality:main3} 
   \frac{a}{k} <\frac{\sigma_k(M)}{\sigma_k(\Gamma,V_\Sigma)}<b.
 \end{gather}
\end{theorem}

\subsubsection*{Comments and discussion}

Coarse discretizations have been used for a long time in spectral
geometry of the Laplace operator on closed Riemannian manifolds. They
were used by Buser \cite{buser1} to construct compact
hyperbolic surfaces with large area and uniformly positive
$\lambda_2$. They were 
also used by Brooks \cite{brooks1} to study the first non-zero eigenvalue of
towers of covering.  For the eigenvalues of the Laplace operator, a
result similar to our Theorem \ref{theorem:MainSpectralComparison}
appeared in the work of Mantuano  \cite{mantuano}. Our proof will be
in the same spirit, but some serious technical difficulties occur
because of the important role 
played by the boundary. As in \cite{mantuano}, many of the tools that
we use are from Chavel's book \cite{cha2}, particularly from Section
VI.5.

It is important not to confuse this type of discretization with those
used in numerical analysis. Our goal is to dicretize in a coarse
sense, which is not sensitive to the local geometry. In particular,
the interesting applications of our method are performed using a fixed
value of the ``mesh parameter'' $\epsilon\in(0,r_0/4)$. It is worth observing that any
compact manifold is roughly isometric to a point. The emphasis is on
the control of the constants $a,b,\tau$. In fact, if we let
$\epsilon\rightarrow 0$, the control of the constants is lost: the
manifold is not approximated in a better way for smaller values of
$\epsilon$. Note also that in the present context any $\epsilon$-discretization is a finite graph.

\subsection{Applications and examples}
We will give three applications of our method to the construction of
sequences of surfaces with large Steklov eigenvalues. Each of these
are in the same spirit: a graph $G=(V,E)$ will describe a pattern to be
used in the construction of a corresponding surface
$\Omega_G$. Roughly speaking, a finite set $D_1,\cdots,D_m$ of
\emph{fundamental pieces} is given. These are used to build a surface
by associating a copy of one of the $D_i$'s to each vertex, and the
graph structure prescribes the pattern to follow for gluing the
various fundamental pieces together. It is often natural to expect the
geometric and spectral properties of the initial graph to be related
to those of its induced surface. We will consider a sequence of graphs
which displays some "spectral expansion" and show using our
discretization results (Theorem \ref{theorem:MainSpectralComparison}
and Proposition \ref{proposition:spectrumRoughlyIsomGraphs}) how to
transplant these expansion properties to the induced surfaces. This
method is classical. See \cite{cg1, colbmat, brooks1}. Related methods
were also introduced in \cite{post}.

\subsubsection*{Application 1} 
The following surprising fact follows from Theorem 1.3 of \cite{ceg2}:
let $\Omega_l \subset \mathbb R^n$ be a sequence of domains with
smooth boundary $\Sigma_l$, with $l\in\N$. If $n \ge 3$ and if the
isoperimetric ratio
$$I(\Omega_l):=\frac{\mbox{Vol}_{n-1}(\Sigma_l)}{\mbox{Vol}_n(\Omega_l)^{\frac{n-1}{n}}}$$
tends to $\infty$, then the normalized Steklov eigenvalues
$\sigma_k(\Omega_l)\mbox{Vol}_{n-1}(\Sigma_l)^{1/(n-1)}$ tend to $0$ as $l\nearrow\infty$.
We will prove that the condition $n\geq 3$ is necessary.

\begin{theorem}\label{theorem:planardomains} 
There exists a sequence of planar domains $\Omega_l \subset \mathbb R^2$ with
smooth boundary $\Sigma_l$, with $l\in\N$, such that
\begin{enumerate}
\item
The isoperimetric ratio $I(\Omega_l) \to \infty$ as $l \to \infty$;

\item
There exists a constant $C>0$ (independant of $l$), such that for each
$l$, $\sigma_2(\Omega_l)|\Sigma_l| \ge C$.
\end{enumerate}
\end{theorem}

\subsubsection*{Application 2}
In Theorem 2 of \cite{ceg2}, it was shown that if $\Omega \subset
\mathbb R^n$ is a domain with smooth boundary $\Sigma$, then
\begin{gather}\label{inequality:ceg2}
  \sigma_k(\Omega) \mbox{Vol}_{n-1}(\Sigma)^{1/(n-1)} \le C(n) k^{2/n},
\end{gather}
where $C(n)$ is a positive constant depending only on the dimension $n$.
As the domains under consideration are Euclidean, it is a natural
question to decide wether or not a similar estimate holds for more
general flat Riemannian manifolds. The following Theorem shows that it
is not the case.
\begin{theorem}\label{theorem:LargeLambdaOneFlat}
  There exists a sequence $\{\Omega_l\}_{l\in\mathbb{N}}$ of compact
  flat Riemannian surfaces with boundary $\Sigma_l$
  and a constant $C>0$ (independant of $l$) such that for each $l\in\mathbb{N}$,
  $\mbox{genus}(\Omega_l)=1+l,$ and
  $$\sigma_2(\Omega_l)L(\Sigma_l)\geq Cl.$$
\end{theorem}

\begin{remark}
One could use the present method to give an alternative proof of
Theorem 1 from \cite{cg1}, or a version of Theorem
\ref{theorem:LargeLambdaOneFlat} for surfaces of constant curvature
$-1$.
\end{remark}

\subsubsection*{Application 3}

In \cite{cg1} two of the authors constructed surfaces modelled on
regular graphs and developped an ad hoc spectral comparison
inequality which allowed the construction of a sequence of surfaces
$\Omega_l$ with boundary $\Sigma_l$ such that
\begin{gather*} 
  \lim_{l\rightarrow\infty}\sigma_2(\Omega_l)L(\Sigma_l)=+\infty.
\end{gather*}
This sequence also satisfies
$\lim_{l\rightarrow\infty}\mbox{genus}(\Omega_l)=+\infty$, which is a
necessary condition since it is known \cite{kokarev1} that 
$$\sigma_2(\Omega)L(\Sigma_l)\leq 8\pi(\mbox{genus}(\Omega)+1).$$
In the construction proposed in \cite{cg1}, the number of boundary
components of the surface $\Omega_l$ is also proportionnal to $l$. It is natural
to ask if we can reduce the number of boundary components.
By studying this construction in the context of discretizations, we
will prove that it is possible to choose each $\Omega_l$ to have
exactly one boundary component. 
\begin{theorem}\label{theorem:LargeLambdaOneConnectedBdr}
  There exist a sequence $\{\Omega_l\}_{l\in\mathbb{N}}$ of compact
  surfaces with connected boundary and a constant $C>0$ such that
  for each $l\in\mathbb{N}$, $\mbox{genus}(\Omega_l)=1+l,$ and
  $$\sigma_2(\Omega_l)L(\partial\Omega_l)\geq Cl.$$
\end{theorem}

\begin{remark}
The surfaces that we construct in the
above three applications are not necessarily in the class
$\mathcal{M}$, but they are uniformly quasi-isometric to such
manifolds. 
For instance, if a sequence of manifolds $M_l$ is replaced
by manifolds $X_l$ which are quasi-isometric to $M_l$ with the same
constants $a,b$ for each $l\in\N$, then $\sigma_2(M_l)$ tends to 0 if and only if $\sigma_2(X_l)$ does. This will play a crucial role in Section
\ref{section:applications}.
\end{remark}

\subsection{Notations}
Some of the constants appearing in various results will have to be
reused later on. These are numbered successively as
$C_1,C_2,\cdots$. Each $C_j$ is used precisely once in the
paper. These constants 
can depend on the bounds $\kappa, r_0$ and on the dimension $n$ and
parameter $\epsilon>0$. This dependance will not be stated explicitely
each time. 
The symbol $\strokedint$ is used for the averaging operator on its
domain. Given a function $F\in C^\infty(M)$, we write 
$$\|F\|_\Sigma:=\|F\mid_\Sigma\|_{L^2(\Sigma)}$$
We will write $\nabla^\Sigma$ for the gradient operator on $C^\infty(\Sigma)$ and also
$$\|\nabla^{\Sigma}F\|_\Sigma:=\|\nabla^{\Sigma}(F\mid_\Sigma)\|_{L^2(\Sigma)}.$$
When the volume form is clear from the context, it will be omitted.
Given a graph $\Gamma=(V,E)$, the set of all real valued functions on
the vertices $V$ is written $\ell^2(V)$, understood with its natural $\ell^2$
inner product.

\subsection{Plan of the paper}
In Section \ref{section:discretization} we introduce a coarse
discretization of manifolds in the class $\mathcal{M}(\kappa,r_0,n)$
and study its basic properties. This is followed in Section
\ref{section:roughlyisometricgraphs} by the introduction of  Steklov
eigenvalues for graphs with boundary, and a spectral comparison for
roughly isometric graphs is proved in Proposition
\ref{proposition:spectrumRoughlyIsomGraphs}. In the next section
several tools are introduced: a comparison inequality for the
Dirichlet energy of a function and its restriction to the boundary
(Lemma \ref{lemma:FourierStuff}), a local Poincaré inequality on
cylinders (Lemma \ref{lemma:kanaiCylinder}). This is followed by the
introduction of discretization of smooth functions, and the smoothing
of discrete functions. In Section \ref{section:proofmaincomparison}
the proof of the main comparison inequality (Theorem
\ref{theorem:MainSpectralComparison}) is presented. Finally, in the
last section we present the three applications to surfaces with large
Steklov eigenvalue $\sigma_2$.

\subsection{Acknowledgements}
We would like to thank Iosif Polterovich for useful conversations during the preparation of this manuscript. The third author was supported by
\emph{Swiss National Science Foundation} (Proposal no. 200020\_149261). AG acknowledges the support of both \emph{National Sciences and Engineering Research Council of Canada} and \emph{Fonds de recherche du Québec- nature et technologies} for the duration of this project.

\section{Discretization of compact manifolds with boundary}
\label{section:discretization}

In this section, the geometric discretization of the manifold $M$ is introduced. Because we are considering the Steklov problem, the boundary plays a crucial role. It is therefore natural that the discretization will lead to a \emph{graph with boundary}.
\begin{defi}
  A graph with boundary is a pair $(\Gamma, B)$ where $\Gamma=(V,E)$
  is a graph and $B\subset V$ is a distinguised set of vertices. 
  The path-distance on $V$ is defined as follows: given
  $x,y\in V$, the distance $d_{\Gamma}(x,y)$ is the length of the
  shortest path between $x$ and $y$, where two adjacent vertices are
  at distance $1$. There is a natural graph structure on $B$ defined
  by $E_B \subset E$, where $e \in E_B$ iff the edge $e$ joins two vertices of
  $B$. However, on the graph $(B,E_B)$, we consider the extrinsic
  distance $d_{\Gamma}(x,y)$.
\end{defi}

Given $0<\epsilon<r_0/4$, let $V_\Sigma$ be a maximal
$\epsilon$-separated set in $\Sigma$. Let $V_\Sigma'$ be a copy of
$V_\Sigma$ located $4\epsilon$ away from the boundary:
\begin{gather*}
  V_\Sigma':=\{4\epsilon\}\times V_\Sigma\subset M.
\end{gather*}
Let $V_I$ be a maximal $\epsilon$-separated set in $M\setminus
[0,4\epsilon)\times \Sigma$ such that $V_\Sigma'\subset V_I.$ (See Figure~\ref{fig:discretization}.)
\begin{figure}[ht]
  \centering
  \includegraphics[width=6cm]{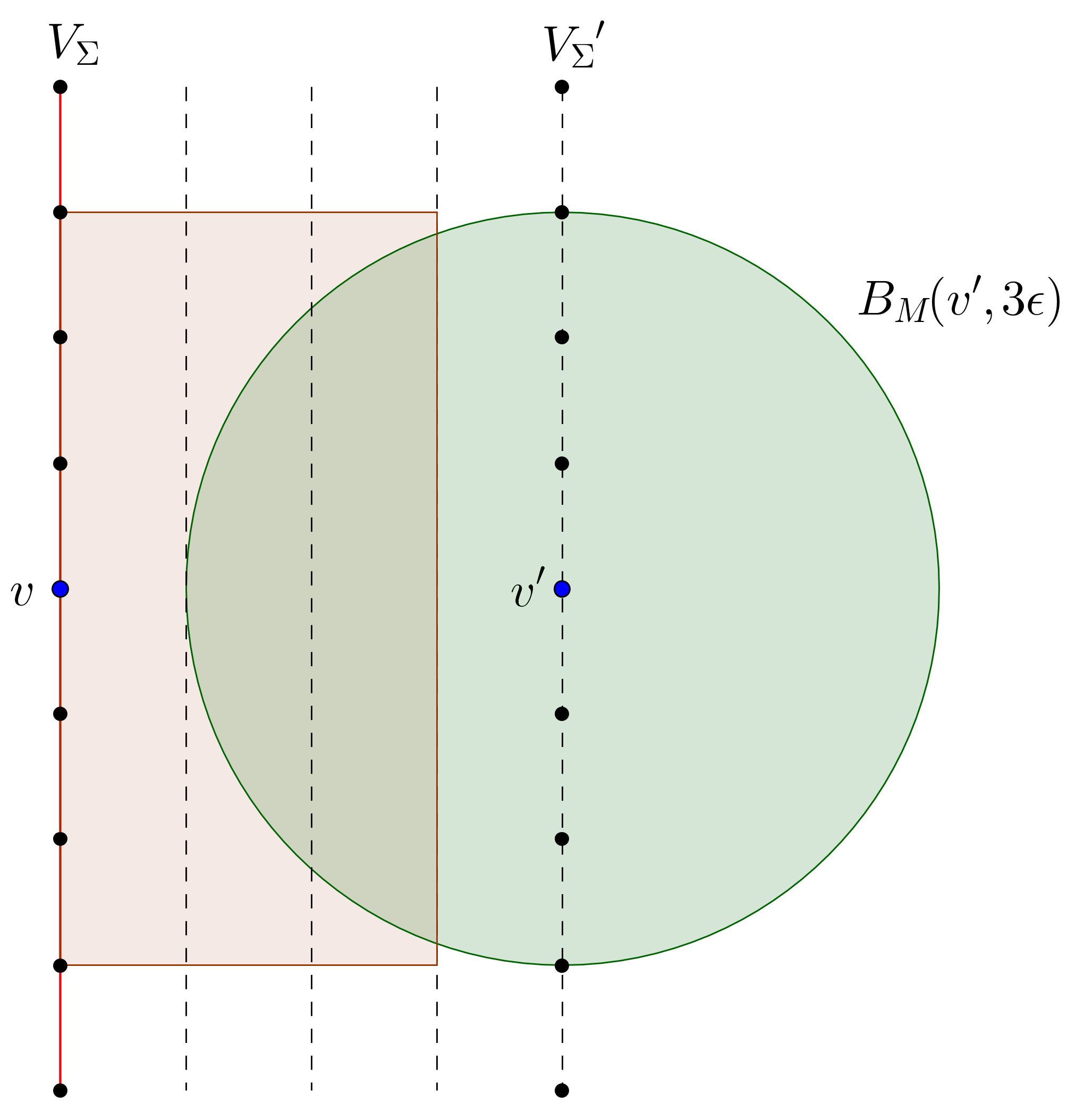}
  \caption{The discretization close to the boundary}
  \label{fig:discretization}
\end{figure}
The set $V=V_\Sigma\cup V_I$ is
given the structure of a graph $\Gamma$  by declaring
\begin{itemize}
\item any two $v,w\in V$  adjacent whenever
  $d_M(v,w)<3\epsilon$;
\item any $v\in V_\Sigma$ adjacent to $v'=(4\epsilon,v)\in
  V_\Sigma'\subset V_I$.
\end{itemize}
The graph $\Gamma=(V,E)$ together with boundary $B=V_\Sigma$ is called an \emph{$\epsilon$-discretization of $M$}. The path-distance on
$\Gamma$ is $d_\Gamma$.

\begin{remark}
  The
  key features of this discretization are:
  \begin{itemize}
  \item The regions $C_v:=[0,3\epsilon)\times B_\Sigma(v,3\epsilon)$ and $B_M(v,3\epsilon)$ (which are shaded in Figure~\ref{fig:discretization}) have a large
    enough intersection. This will be important in the proof of
    Lemma~\ref{lemma:discretizationenergybound}, which controls the
    energy of the discretization of a function.
  \item   Because the interior vertices are separated by a strip of
  width $4\epsilon$ from the boundary,
  $B_M(v,4\epsilon)\cap\Sigma=\emptyset$ for each $v\in V_I$. This
  will allow the use of Kanai's inequality for functions 
  $f\in C^\infty(B_M(v,4\epsilon))$ (See Lemma \ref{lemma:kanaiCylinder}).
  \item The interaction between the boundary of the graph and its
    interior is simple, and it reflects the product structure of the
    manifold $M$. This will lead to the existence of a suitable
    partition of unity, in Section~\ref{section:smoothing}.
  \end{itemize}
\end{remark}

\begin{remark}
  It follows from the Bishop--Gromov theorem that the degree
  of each vertex   $v\in V$ is bounded above in terms of
  $\kappa$ and $\epsilon$. The total number of vertices
  is also be controlled. See \cite[p. 147]{cha2}. 
\end{remark}

\begin{lemma}\label{lemma:discretizationsRoughlyIso}
  For any $0<\epsilon<r_0/4$, and any $\epsilon$-discretization $(\Gamma,V_\Sigma)$ of $M$, the natural inclusion $V\subset M$ is a rough isometry. Indeed, the following stronger estimate holds:
  $$\frac{\epsilon}{4}d_\Gamma(x,y)-10\leq d_M(x,y)\leq 4\epsilon d_\Gamma(x,y).$$
\end{lemma}

\begin{proof}
Let $x,y\in V$ with $d_{\Gamma}(x,y)=k$. This means that there exists a sequence of vertices
$$x=x_0,x_1,...,x_k=y \in V$$ 
with $d_{\Gamma}(x_i,x_{i+1})=1$ which represents a shortest path between $x$ and $y$.
By construction of the discretization, we have $d_M(x_i,x_{i+1})\le 4\epsilon$, and by compactness, there exists a path on $M$ between $x_i$ and $x_{i+1}$ of length $\le 4\epsilon$. By concatenation of these paths, we find a path joining $x$ and $y$ on $M$ whose length $\le 4\epsilon k$. This implies
$$
d(x,y)\le 4 \epsilon k=4\epsilon d_{\Gamma}(x,y),
$$
which completes the proof of the second inequality.

To prove the left-hand-side inequality, consider $x,y\in V$ with $d(x,y)=\alpha$. There are three cases to consider.

\medskip
\noindent
\textbf{Case 1}: $x,y \in V_I$, that is $x$ and $y$ are not points of the boundary $B=V_{\Sigma}$ of the graph $\Gamma$.
Because of the product structure near the boundary, there exists a geodesic parametrised by arc length
$$
\gamma:[0,\alpha] \rightarrow M
$$ 
between $x$ and $y$ on $M$ whose length is $\alpha$. It follows from convexity that 
$$Im(\gamma)\subset M\setminus [0,4\epsilon)\times \Sigma:$$
a geodesic entering $[0,4\epsilon)\times \Sigma$ would be trapped.

Let $k\in\N$ be such that $k-1 < \frac{\alpha}{\epsilon} \le k$ and define $t_i=\frac{i\alpha}{k}$, $i=0,...,k$.
and $x_i=\gamma(t_i) \in M$. Note that $x_0,x_k \in V$, but that this is in general not the case for the other points $x_i$. However, by construction, $\bigcup_{x\in V_I} B(x, \epsilon) = M\setminus
[0,4\epsilon)\times \Sigma$ so that there exist $y_0=x_0,y_1,...,y_{k-1},y_k=y \in V$ with $d(y_i,x_i) \le \epsilon$ and
$$
d(y_{i+1},y_i) \le d(y_{i+1},x_{i+1})+d(x_{i+1},x_{i})+d(y_{i},x_{i}) \le 3 \epsilon
$$
so that $y_i$ and $y_{i+1}$ are connected in $\Gamma$ and $d_{\Gamma}(y_i,y_{i+1})\le 1$.
It follows that
$$
d_{\Gamma}(x,y)\le k\le (\frac{\alpha}{\epsilon}+1) =\frac{1}{\epsilon} d(x,y)+1.
$$

\medskip
\noindent
\textbf{Case 2}: $x \in B=V_{\Sigma}$ and $y\not \in B$: this means that there is exactly one point $x_1 \in V_{\Sigma}'$ which is connected to $x_0=x$ and $d_{\Gamma}(x_0,x_1)=1$.
Now, $x_1$ and $y$ are as in the first step, so that we have
\begin{align*}
d_{\Gamma}(x,y)&\le d_{\Gamma}(x,x_1)+d_{\Gamma}(x_1,y)\\
&\le 1+\frac{1}{\epsilon} d(x_1,y)+1\\
&\le 2+\frac{1}{\epsilon}(d(x_1,x)+d(x,y))=2+\frac{1}{\epsilon}(4\epsilon+d(x,y))=6+\frac{1}{\epsilon}d(x,y).
\end{align*}

\medskip
\noindent
\textbf{Case 3}: $x,y \in B=V_{\Sigma}$: we do as in the second step and get 
$$
d_{\Gamma}(x,y) \le 10+\frac{1}{\epsilon}d(x,y).
$$

\smallskip
\noindent
End of the proof: we have to show that $\bigcup_{x\in V} B(x, 4\epsilon) = M$. We already know that this is true for points in $M\setminus [0,4\epsilon)\times \Sigma$ because $V_I$ is a maximal separated set in this space. 

If $x\in [0,4\epsilon)\times \Sigma$, there is $y \in \Sigma$ or $y \in \{4\epsilon\} \times \Sigma$ with $d(x,y)\le 2\epsilon$. By construction, there is a point $z\in V_{\Sigma}$ in the first case, or a point $z\in V_{\Sigma}'$ in the second case such that $d(y,z)\le \epsilon$, and we deduce that $d(z,x) \le 4 \epsilon$.
   
\end{proof}

\section{Spectrum of roughly isometric graphs}
\label{section:roughlyisometricgraphs}

Let $(\Gamma,B)$ be a graph with boundary. The \emph{Dirichlet energy}
of a function 
$f:V\rightarrow\mathbb{R}$ is  
$$q(f):=\sum_{v\sim w}\bigl(f(v)-f(w)\bigr)^2.$$
To simplify notations, we will write $\|f\|_B$ for $\|f\mid_B\|_{\ell^2(B)}$. Similarly, 
$$q_B(f):=\sum_{v\sim w, v,w\in B}\bigl(f(v)-f(w)\bigr)^2.$$

\begin{defi}\label{defi:steklovgraph}
For each $j=1,\cdots, |B|$, the $j$-th \emph{Steklov eigenvalue} of
$(\Gamma,B)$ is defined by
\begin{gather}\label{discretesteklov:variational}
  \sigma_j(\Gamma,B)=\min_E\max_{f\in E}
  \frac{q(f)}{\|f\|_{B}^2}.
\end{gather}
where the minimum is over all $j$-dimensional linear subspaces $E$ of
$\ell^2(V)$.
\end{defi}
In particular, $\sigma_1(\Gamma,B)=0$ is realized by locally constant functions on $V$.

\begin{remark}\label{remark:fullvertices}
  Note that for $B=V$, $\sigma_j=\lambda_j(L)$ is the $j$-th
  eigenvalue of the graph Laplacian $L$. See \cite{mantuano}.
\end{remark}

The notion of rough isometry between metric spaces was presented in Definition \ref{definition:roughisometry}. A specialized version will be useful.
\begin{defi}
  A rough isometry $\Phi$ between two graphs with boundary
  $(\Gamma_1,B_1)$ and $(\Gamma_2,B_2)$ is a rough isometry of the
  underlying graphs which sends $B_1$ to $B_2$. In other words, the restriction of $\Phi$ to $B_1$ is a rough isometry $B_1\rightarrow B_2$ when considering extrinsic distances on $B_1$ and $B_2$.
\end{defi}

The easiest situation in which spectral comparison occur is for rough
isometries between graphs. This will be useful for applications in Section~\ref{section:applications}.
\begin{prop}\label{proposition:spectrumRoughlyIsomGraphs}
  Given $a\geq 1$ and $b,\tau\geq 0$, 
  there exist constants $A,B$ depending only on $a,b,\tau$ and on the maximal degree of vertices, such that
  any two graphs with boundary $(\Gamma_1,B_1)$ and $(\Gamma_2,B_2)$
  which are roughly isometric (through $\Phi$) with constants $a,b,\tau$ satisfies
  $$
  A \leq \frac{\sigma_k(\Gamma_1,B_1)}{\sigma_k(\Gamma_2,B_2)} \leq B
  $$
  for each $k \leq \min\{|B_1|,|B_2|\}$.
\end{prop}
This proposition is in the same spirit as Theorem 2.1 of
\cite{mantuano}. Nevertheless, the presence of a boundary brings many
new difficulties. For instance, the following inequality for
$f:V(\Gamma_2)\rightarrow\R$ was proved in \cite[Lemma VI.5.4]{cha2}
\begin{equation*}
    \| f\|^2 _{B_2} \leq \| f\|_{\Gamma_2} ^2 \leq Cq(f) + C'\|\Phi^*f \|_{\Gamma_1} ^2,
\end{equation*}
for constants $C$ and $C'$ which
only depend on $a,b,\tau$ and the maximal degree of both graphs. Here and elsewhere in the paper, $\Phi^*f=f\circ\Phi$ is the pullback of $f$ by $\Phi$.
It is used in \cite{mantuano} to obtain a lower bound on $\|\Phi^*f \|_{\Gamma_1}^2$.
Because we consider Steklov eigenvalues, a lower bound on  $\|\Phi^*f\|_{B_1}^2$ is needed.
\begin{lemma}\label{lemma:chavelabord}
  Given a rough isometry $\Phi:(\Gamma_1,B_1)\rightarrow (\Gamma_2,B_2)$
  between two graphs with boundary,  there exist constants $\Cl{chavelbord1},\Cl{chavelbord2}$
  which depend only on the constants $a,b,\tau$ of the quasi-isometry
  and on the maximal degree of the graphs, such that any function
  $f:V(\Gamma_2)\rightarrow\R$ satisfies
  \begin{equation}\label{eq5}
    \| f\|^2 _{B_2} \leq \Cr{chavelbord1}q(f) + \Cr{chavelbord2}\|\Phi^*f \|_{B_1}^2.
  \end{equation}
\end{lemma}

\begin{proof}[Proof of Lemma \ref{lemma:chavelabord}]
We have to adapt Lemma VI.5.4 of \cite{cha2} to graphs with boundary. The
rough isometry $\Phi$ has a rough inverse (See the introduction of
\cite{kanai} for a discussion.) That is a rough isometry $\Psi$  from
$(\Gamma_2,B_2)$ to $(\Gamma_1,B_1)$, with $\Psi(B_2)\subset B_1$ such
that any $x\in V_1$ and $y\in V_2$ satisfy
$$
d(x,\Psi \circ \Phi(x))\le K;\quad d(y,\Phi \circ \Psi (y))\le K
$$
where $K$ is a constant depending on the constants $a,b,\tau$ of the rough isometry $\Phi$.
For $y\in B_2$, 
\begin{gather}\label{equation:fcarre}
f^2(y)\le 2(f(y)-(\Phi \circ \Psi)^*f(y))^2+ 2(\Phi \circ \Psi)^*f(y)^2.
\end{gather}
In order to bound $\sum_{y\in B_2} f(y)^2$, the two terms on the right-hand side of this inequality will be estimated.
As $d(y,\Phi \circ \Psi(y))\le K$, there is a path $y_1=y,y_2,...,y_n=\Phi \circ \Psi(y)$ of length at most $K$, and it follows from triangle and Cauchy-Schwartz Inequality that
$$
( (f(y)-(\Phi \circ \Psi)^*f(y))^2 \le K^2\sum_{i=1}^{n-1}(f(y_{i+1})-f(y_i))^2.
$$
This sum only involves points that are at a distance at most $K$ from $y\in B_2$. As the number of points in a ball of radius $K$ is bounded above in terms of the maximal degree, there exists a constant $\Cr{chavelbord1}$ (depending only on $K$ and on the maximal degree) such that 
$$2\sum_{y\in B_2}(f(y)-(\Phi \circ \Psi)^*f(y))^2\leq C(K)q(f).$$

We also have to estimate $\sum_{y\in {B_2}} ((\Phi \circ \Psi)^*f)^2(y)$. Observe that
$$
\sum_{y\in {B_2}} ((\Phi \circ \Psi)^*f)^2(y)=\sum_{y\in {B_2}} (\Phi^*f)^2(\Psi(y)).
$$
But $\Psi(y) \in B_1$, which may be the image of at most a finite controlled number of vertices $y \in B_2$. Therefore there exist another constant $\Cr{chavelbord2}$ such that
$$
2\sum_{y\in B_2} ((\Phi \circ \Psi)^*f)^2(y)\le \Cr{chavelbord2} \sum_{z\in B_1} (\Phi^*f)^2(z).
$$
The results follows by substitution of the previous two inequalities in \eqref{equation:fcarre}.
\end{proof}

The previous Lemma will be the main tool used in the following proof.
\begin{proof}[Proof of
  Proposition~\ref{proposition:spectrumRoughlyIsomGraphs}]
 Let $\Phi$ be a rough isometry between $(\Gamma_1,B_1)$ and
  $(\Gamma_2,B_2)$, with constants $a,b,\tau$.
  Given a function $f: V_2 \rightarrow\R$, define $\Phi^*f : V_1\rightarrow\R$ by $\Phi^*f(x)=f(\Phi(x))$. 
  The following inequality follow from Lemma VI.5.2 of \cite{cha2}:
  \begin{equation}\label{eq3}
    q(\Phi^*f) \leq \Cl{roughisometry1} q(f),
  \end{equation}
	where $\Cr{roughisometry1}$ depend on the maximal degree and on $a,b,\tau$.
	
	\medskip
  We will treat separately the situations where $\sigma_k(\Gamma_2)$ is smaller or larger than $\frac{1}{2\Cr{chavelbord1}}$.

\noindent
\textbf{Situation 1}:  Suppose that $\sigma_k(\Gamma_2) \leq \frac{1}{2\Cr{chavelbord1}}$.

  Let $f_1, \cdots , f_k$ be eigenfunctions corresponding to
  $\sigma_1(\Gamma_2,B_2), \cdots, \sigma_k(\Gamma_2,B_2)$
  respectively. 
  The space $E_k:=\mbox{span}(\Phi^*f_1,\cdots,\Phi^*f_k)$ will be
  used in the variational characterization of $\sigma_k(\Gamma_1,B_1).$ So, if $f\in E_k$ and $g= \Phi^{*}f$, we have in particular to show that the restriction $g$ to the boundary $B_1$ of $\Gamma_1$ is large enough. More precisely, we have by (\ref{eq5}),
 $$
  \|g\|_{B_1}^2 
  =   \|\Phi^*f \|_{B_1}^2 \geq 
  \Cr{chavelbord2}^{-1}\left(\| f\|^2 _{B_2} - \Cr{chavelbord1}q(f)  \right) 
  \geq \Cr{chavelbord2}^{-1}\|f\|^2 _{B_2}\left( 1 - \Cr{chavelbord1}\sigma_k(\Gamma_2)  \right) \geq \frac{1}{2\Cr{chavelbord2}}\|f\|^2 _{B_2} > 0.
  $$
  Using \eqref{eq3} and the above inequality leads to $\frac{\sigma_k(\Gamma_1)}{\sigma_k(\Gamma_2)} \leq 2\Cr{roughisometry1}\Cr{chavelbord2}$. 
  
	\medskip
	\noindent
	\textbf{Situation 2}:  Suppose that $\sigma_k(\Gamma_2) \geq \frac{1}{2\Cr{chavelbord1}}$. 
	
	In this case, 
  $\sigma_k(\Gamma_1) \leq
  2\,\Cr{chavelbord1}\sigma_k(\Gamma_1)\sigma_k(\Gamma_2)$ and
  $$\frac{\sigma_k(\Gamma_1)}{\sigma_k(\Gamma_2)} \le 2\Cr{chavelbord1} \sigma_k(\Gamma_1).$$ 
  The conclusion now follows from  the fact that $\sigma_k(\Gamma_1)$ is bounded above uniformly in terms of the maxiamal degree of the graph.
\end{proof}

\begin{remark}
This proof is typical. In fact, the proof of Theorem \ref{theorem:MainSpectralComparison} has a similar structure to that of Proposition \ref{proposition:spectrumRoughlyIsomGraphs}. Nevertheless, the techniques are much more involved in this latter case.
\end{remark}

\section{Preliminary  results}
\label{section:preliminary}

In this section some results which will be used in the proof of
Theorem~\ref{theorem:MainSpectralComparison} are presented.

\subsection{The Dirichlet energy on a manifold and on its boundary}
On a graph with boundary $(\Gamma,B)$ one immediately sees that
$$q_B(f):=\sum_{\stackrel{v\sim w}{ v,w\in B}}\bigl(f(v)-f(w)\bigr)^2\leq q(f).$$
That is, restricting a function to the boundary reduces its energy.
There is no such simple formula for compact manifolds with
boundary. Nevertheless, under more restrictive hypothesis, some
control can still be granted. 
\begin{lemma}\label{lemma:FourierStuff}
  For each Steklov eigenfunction $F\in C^\infty(M)$ corresponding to $\sigma<1/4$, the following holds:
  \begin{gather}
    \|\nabla^\Sigma F\|_{\Sigma}^2 \leq \frac{1}{2}\|\nabla F\|_{M}^2.  
  \end{gather}
\end{lemma}
\begin{proof}
  Let $(f_k)\subset L^2(\Sigma)$ be an orthonormal basis corresponding
  to the eigenvalues $\lambda_k$ of the Laplacian on $\Sigma$.
  Let $F\in C^\infty(M)$ be a Steklov eigenfunction corresponding to
  $\sigma$. On the cylindrical neighborhood $\Sigma\times [0, 1]$ the
  Fourier decomposition of $F$ is
  $$F = \sum_{k = 0}^{\infty} a_k(r)f_k(\theta)\qquad r\in
  [0,1],\quad \theta\in\Sigma.$$ 
  As the function $F$ is harmonic on $M$, and hence on the cylinder,
  the following holds 
  $$
  \Delta F(r,\theta) = \sum_{k=0}^\infty[ -a_k''(r)f_k(\theta) + a_k(r)\lambda_kf_k(\theta)] = 0
  $$
  which implies that $a_k''(r) = \lambda_k a_k(r)$ and hence 
  $$
  a_k(r) = a_k(0)\cosh(\sqrt{\lambda_k}r) + \frac{1}{\sqrt{\lambda_k}}a_k'(0)\sinh(\sqrt{\lambda_k}r).
  $$
  Moreover, on the boundary $\Sigma$,
  $$0= \sigma F(\theta)-\frac{\partial F}{\partial n}(\theta)= \sum_k
  \bigl(\sigma  a_k(0) + a_k'(0)\bigr)f_k(\theta),$$
  which implies $\sigma a_k(0)+a_k'(0)=0$ whence by substitution
  leads to
  $$
  a_k(r) = a_k(0)[\cosh(\sqrt{\lambda_k}r) - \frac{ \sigma}{\sqrt{\lambda_k}}\sinh(\sqrt{\lambda_k}r)]
  $$
  and 
  $$
  a_k'(r) = a_k(0)\sqrt{\lambda_k}[\sinh(\sqrt{\lambda_k}r) - \frac{ \sigma}{\sqrt{\lambda_k}}\cosh(\sqrt{\lambda_k}r)]
  $$
  The Dirichlet energy on the boundary and on the cylinder are
  expressed by
  $$
  \|\nabla^{\Sigma}F\|_{\Sigma}^2 = \sum_k a_k^2(0)\lambda_k
  $$
  and 
  \begin{gather}\label{e1}
    \|\nabla F\|_{\Sigma \times (0, 1)}^2 = \sum_k \int_0^1 [(a_k')^2 + a_k^2 \lambda_k]dr.
  \end{gather}
  Now using $x=\sqrt{\lambda_k}$ leads to
  $$
  a_k^2(r) = a_k^2(0)\left[\cosh^2(xr) + \frac{ \sigma^2}{x^2}\sinh^2(xr) 
    - \frac{ \sigma}{x}\sinh(2xr)\right]
  $$
  and 
  $$
  (a_k^{'})^2(r) = a_k^2(0)x^2\left[\sinh^2(xr) + \frac{ \sigma^2}{x^2}\cosh^2(xr)
    - \frac{ \sigma}{x}\sinh(2xr)\right]
  $$
  Substitution in equation~\eqref{e1} and evaluation of the integrals
  give 
 \begin{eqnarray*}
    \|\nabla F\|_{\Sigma \times (0, 1)}^2 & = & \sum_k a_k^2(0)x^2 \int_0^1 \left[\left(1 + \frac{ \sigma^2}{x^2}\right)
      \cosh(2xr) - \frac{2 \sigma}{x}\sinh(2xr)\right]dr \\
    & = &  \sum_k a_k^2(0)x^2 \left[\left(1 + \frac{ \sigma^2}{x^2}\right) \frac{\sinh(2x)}{2x}
 - \frac{\sigma}{x^2}\left(\cosh(2x) - 1\right)\right].
\end{eqnarray*}
Moreover,
for $\sigma<1/4$, it follows from $\tanh(x)\leq x\leq x/4\sigma$ that $\cosh(x)-\frac{2\sigma}{x}\sinh(x)\geq \frac{1}{2}\cosh(x),$
whence
\begin{align*}
  \left(1 + \frac{ \sigma^2}{x^2}\right) \frac{\sinh(2x)}{2x}
  - \frac{\sigma}{x^2}\left(\cosh(2x) - 1\right)
  &=\frac{\sinh(x)}{x}
  \left(
    \cosh(x)-2\frac{\sigma}{x}\sinh(x)+\frac{\sigma^2}{x^2}\cosh(x)
  \right)\\
  &\geq\frac{\sinh(x)}{x}\left(\frac{1}{2}\cosh(x)+\frac{\sigma^2}{x^2}\right)\\
  &=\left(\frac{\sinh(x)\cosh(x)}{2x}+\frac{\sigma^2}{x^2}\frac{\sinh(x)}{x}\right)\geq\frac{1}{2}.
\end{align*}
\end{proof}

In order to use the above Lemma in estimations of higher eigenvalues, one needs to consider linear combinations of eigenfunctions.

\begin{cor}\label{coro:linearcombiFourrierStuff}
Let $F\in C^\infty(M)$ be a linear combination of the first $k$ Steklov eigenfunctions $F_1,\cdots,F_k$. If $\sigma_k \le \frac{1}{4}$ then
$F=a_1 F_1+\cdots+ a_k F_k$
satisfy the following:
$$\|\nabla^{\Sigma}F\|_{\Sigma}^2 \le \frac{k}{8}\|\nabla F\|_{M}^2.$$
\end{cor}

\begin{proof}
 It follows from Green's formula that
$$
\int_M\langle \nabla F_i,\nabla F_j\rangle =\int_M \Delta F_iF_j+\int_{\Sigma} \partial_{\nu}F_iF_j=
\int_{\Sigma} \sigma_i F_i F_j,
= \sigma_i \delta_{ij}.
$$
It follows from the Cauchy-Schwarz inequality that
\begin{align*}
\|\nabla^{\Sigma}F\|_{\Sigma}^2&=\vert \int_{\Sigma} \langle \nabla^{\Sigma}F,\nabla^{\Sigma}F \rangle\vert\\
&=
\vert \sum_{i,j=1}^k a_ia_j\int_{\Sigma} \langle \nabla^{\Sigma}F_i,\nabla^{\Sigma}F_j\rangle\vert\\
&\le
\vert \sum_{i,j=1}^k a_ia_j\left(\int_{\Sigma} \langle \nabla^{\Sigma}F_i,\nabla^{\Sigma}F_i\rangle\right)^{1/2}
\left(\int_{\Sigma} \langle \nabla^{\Sigma}F_j,\nabla^{\Sigma}F_j\rangle\right)^{1/2}\vert\\
&\le
\sum_{i,j=1}^k \vert a_i\vert \vert a_j \vert \|\nabla^{\Sigma}F_i\|_{\Sigma}\|\nabla^{\Sigma}F_j\|_{\Sigma}.
\end{align*}
It follows from Lemma \ref{lemma:FourierStuff} that 
$$
\|\nabla^{\Sigma}F\|_{\Sigma}^2\le \frac{1}{2}  \sum_{i,j=1}^k \vert a_i\vert \vert a_j \vert \|\nabla F_i\|_{M}\|\nabla F_j\|_{M} =
 \frac{1}{2}  \sum_{i,j=1}^k \vert a_i\vert \vert a_j \vert\sqrt{\sigma_i}\sqrt{\sigma_j}.
$$
We also have
$$
\|\nabla F\|_{M}^2=\sum_{i,j=1}^k a_ia_j \langle \nabla F_i,\nabla F_j\rangle=\sum_{i=1}^ka_i^2\sigma_i.
$$
Now, in general, for $\alpha_i, \alpha_j\geq 0$, it follows from the Cauchy-Schwarz inequality that
$$
\sum_{i,j=1}^k \alpha_i \alpha_j=(\sum_i\alpha_i)^2\leq k\sum_i\alpha_i^2.$$
Setting $\alpha_i=\vert a_i \vert \sqrt{\sigma_i}$, and using $\sigma_i \le \frac{1}{4}$ it follows that
$$
\|\nabla^{\Sigma}F\|_{\Sigma}^2 \le \frac{k}{2} \sum_{i=1}^k a_i^2 \sigma_i \le \frac{k}{8}\sum_{i=1}^ka_i^2\sigma_i=\frac{k}{8} \|\nabla F\|_{M}^2.
$$
\end{proof}

\subsection{Local Poincaré-type inequality on products}
The following Lemma is similar to Lemma 8 of \cite{kanai1986}. 
See also Lemma Vi.5.5 in \cite[p. 177]{cha2}.
\begin{lemma}\label{lemma:kanaiCylinder}
  For each $\delta>0$, there exists a constant
    $\Cl{kanaicyl}=\Cr{kanaicyl}(n,\kappa,\delta)$  with the following properties:
    Given $p\in\Sigma$, let
  $C=B_\Sigma(p,\delta)\times[0,\delta]$.
  Then any smooth function $F\in C^\infty(\overline{C})$ satisfies
  $$\int_{C} |F - F_C| \leq \Cr{kanaicyl}\int_C|\nabla F|,$$
  where $F_C=\strokedint_CF$ is the average of $F$ on $C$.
\end{lemma}

\begin{proof}[Proof of Lemma~\ref{lemma:kanaiCylinder}]
Using $(x,t)\in C$ as coordinates, the integral is split
\begin{align}\label{ineq:proofkanaicyl}
  \int_{C}|F - F_C| &\leq 
  \int_C\left|F(x,t)-\strokedint_0^{\delta}F(x,s)\,ds\right|+
  \int_C\left|\strokedint_0^{\delta}F(x,s)\,ds -
    F_C\right|\\\nonumber
  &=
  \int_B\int_0^\delta\left|F(x,t)-\strokedint_0^{\delta}F(x,s)\,ds\right|+
  \int_0^\delta\int_B\left|\strokedint_0^{\delta}F(x,s)\,ds -
      F_C\right|.
  \end{align}
  We will estimate the two terms in the right-hand side of this inequality separately.
  It follows from Kanai's inequality (Lemma 8 in \cite{kanai1986}) that
  $$\int_0^\delta\left|F(x,t)-\strokedint_0^{\delta}F(x,s)\,ds\right|
  \leq \Cr{kanaicyl}\int_0^\delta\left|\partial_tF\right|.$$
  Moreover, the average of the function $G:B\rightarrow\R$ defined by
  $$G(x)=\strokedint_0^{\delta}F(x,s)\,ds$$
  is $F_C$. Therefore, it follows (again from Kanai's inequality) that
  $$\int_B\left|\strokedint_0^{\delta}F(x,s)\,ds - F_C\right|
  =\int_B\left|G - F_C\right|\leq \Cr{kanaicyl}\int_B|\nabla^\Sigma G|=\Cr{kanaicyl}\int_B\left|\strokedint_0^{\delta}\nabla^\Sigma F\right|.$$ 
  Substitution in \eqref{ineq:proofkanaicyl} now leads to
  \begin{align*}
    \int_{C}|F - F_C|\leq\Cr{kanaicyl}\int_C\left|\partial_tF\right|+\left|\nabla^\Sigma F\right|.
\end{align*}
\end{proof}

\subsection{Discretization of smooth functions}
\label{section:DiscretizationSmoothing}


Let $M\in\mathcal{M}(\kappa, r_0,n).$ Given
$\epsilon\in (0,r_0/2)$ let $\Gamma$ be an $\epsilon$-discretization
of $M$ with boundary $V_\Sigma$. 
The \emph{discretization} $f=DF:V\rightarrow\mathbb{R}$
of a smooth function $F\in C^\infty(M)$ is defined as
\begin{gather}
  f(v) = 
  \begin{cases}
    \strokedint_{B_{\Sigma}(v,3\epsilon)}F & \text{if} \ \  v \in V_\Sigma,\\ 
    \strokedint_{B_{M}(v,3\epsilon)}F & \text{if} \ \ v\in V_I.
  \end{cases}
\end{gather}
The symbol $\strokedint$ is used for the averaging operator on its
domain.
\begin{remark}
  Throughout, we follow Chavel's convention from~\cite{cha2} that
  functions on $M$ are denoted with upper case $F$, while functions on
  (vertices of) the graph $\Gamma$ are denoted with lower case $f$. 
\end{remark}
\begin{lemma}\label{lemma:discretizationenergybound}
  There exists a constant $\Cl{discretizationenergy}$ which only depends on $\kappa,r_0,n$ with the following property. 
  Let $F\in C^{\infty}(M)$ be a linear combination of the first $k$
  Steklov eigenfunctions $F_1,\cdots,F_k$. If $\sigma_k(M)<1/4$, then the discretization $f=DF$ satisfy
  $$q(f) \leq \Cr{discretizationenergy}k\|\nabla F\|_M^2.$$
\end{lemma}
\begin{proof}
  Given a boundary vertex $v\in
  V_\Sigma$, write $v'=(4\epsilon,v)\in V_I$ for the corresponding
  interior vertex. 
  The energy $q(f)$ is the sum of the following three quantities
\begin{gather*}
  E_1= \sum_{\substack{v \sim w \\ v, w \in V_I}} |f(v) - f(w)|^2, \\
  E_2=\sum_{\substack{v \sim w \\ v, w \in V_\Sigma}} |f(v) - f(w)|^2,\\
  E_3= \sum_{v \in V_\Sigma} |f(v) - f(v')|^2.
\end{gather*}

The first term $E_1$ is bounded using the same argument as in
\cite[$\mathcal{D}$:iii,p. 178]{cha2}: there exists a constant $A$ such that
$$E_1\leq A\|\nabla F\|_M^2.$$ 
To bound $E_2$ one also uses \cite[$\mathcal{D}$:iii,p. 178]{cha2} to obtain a
bound in terms of $\|\nabla^\Sigma F\|$ and then in terms of the
interior Dirichlet energy using Corollary \ref{coro:linearcombiFourrierStuff}: there exists a constant $B$ such that
$$E_2\leq Bk\|\nabla F\|_M^2.$$
Let us bound $E_3$. Let $\delta=3\epsilon$. Given $v\in V_\Sigma$,
define $\hat{f}(v)$ to be the average of $F$ on
the cylinder $C_v:=[0,\delta)\times B_\Sigma(v,\delta)\subset M:$
$$\hat{f}(v):=\strokedint_{C_v}F.$$
 It follows that
\begin{align*}
  (f(v) - f(v'))^2 & \leq 2\left((f(v) - \hat{f}(v))^2+(f(v')) - \hat{f}(v))^2\right).
\end{align*}
The fundamental theorem of calculus leads to
\begin{align*}
  |f(v) - \hat{f}(v)|
  & = \left|\strokedint_{C_v}[F(x,0)- F(x,t)] dA(x)dt\right|\\
  & \leq  \strokedint_{C_v}(\int_0^{\delta}
  \left|\partial_rF(x,r)\right|dr)\,dA(x)dt\\
 & = \strokedint_{B_{\Sigma}(v,\delta)} \int_0^{\delta}
 \left|\partial_rF(x,r)\right|dr\,dA(x)
=\frac{1}{|B(p,\delta)|}\int_{C_v}|\partial_rF|.
\end{align*}
The argument used to bound $|f(v')-\hat{f}(v)|$ is similar to that of
$\mathcal{D}$:i,~\cite[p. 178 ]{cha2}. Let
$\beta=|B_M(v',3\epsilon)\cap C_v|$ and observe that
\begin{align*}
  |f(v') - \hat{f}(v)| 
  & = \strokedint_{B_M(v',3\epsilon)\cap C_v}|f(v') - \hat{f}(v)| \\
  & \leq \strokedint_{B_\Sigma(w,3\epsilon)\cap C_v}
  |F-f(v')|+|F - \hat{f}(v)| \\
  &\leq \frac{1}{\beta}\left(\int_{C_v}|F - \hat{f}(v)|+\int_{B_\Sigma(w,2\epsilon)}|F-f(w)| \right).
\end{align*}
These two terms are bounded using Lemma \ref{lemma:kanaiCylinder} and
Kanai's inequality (\cite[p.177]{cha2}), so that
\begin{align*}
  |f(w) - \hat{f}(v)| <
  \frac{1}{\beta}(\Cr{kanaicyl}+\Cl{kanaioriginal})\int_C|\nabla F|.
\end{align*}
The crucial point is that $\beta$ is bounded below in terms of the
geometry. Indeed, let $p=(\{2\epsilon\},v)$. The ball
$B_M(p,\epsilon)\subset B_M(v',3\epsilon)\cap C_v$ and it follows from Croke's inequality [Croke1980] (See also Proposition V.2.3, \cite[136]{cha2}) that
$$\beta\geq |B_M(p,\epsilon)|\geq C'\epsilon^n,$$
where $C'$ is a constant which depends on only on the dimension $n$.
It follows that
$$E_3\leq C\|\nabla F\|_M^2.$$
The bounds on $E_1,E_2$ and $E_3$ lead to
$$q(f)\leq(A+Bk+C)\|\nabla F\|_M^2\leq(A+B+C)k\|\nabla F\|_M^2.$$
The proof is complete, with $\Cr{discretizationenergy}=A+B+C$.
\end{proof}

\subsection{Smoothing of discrete functions}
\label{section:smoothing}

The balls $B(v,3\epsilon)$ for $v \in V$ form an open cover of the manifold $M$. Indeed, it follows from the fact that $V_I$ is a maximal $\epsilon$-separated set in $M\setminus [0,4\epsilon)\times \Sigma$ that 
$$M\setminus [0,4\epsilon)\times \Sigma\subset\bigcup_{x\in V_I} B(x, 3\epsilon).$$
Now, if $x\in [0,4\epsilon)\times \Sigma$, there is $y \in \Sigma$ or $y \in \{4\epsilon\} \times \Sigma$ with $d(x,y)\le 2\epsilon$. But, by construction, there is a point $z\in V_{\Sigma}$ in the first case, or a point $z\in V_{\Sigma}'$ in the second case such that $d(y,z)\le \epsilon$, and we deduce that $d(z,x) \le 3 \epsilon$.

The existence of a partition of unity with controlled energy is a standard tool in geometric analysis. Nevertheless, we could not locate a construction completely adapted to our present context. For the sake of completeness, we therefore proved the following result.
\begin{lemma}
There exists a smooth partition of unity $\{\phi_v\}_{v\in V}\subset C^\infty(M)$ subbordinate to the cover $B(v,4s\epsilon)$ which satisfy the pointwise bound
$$|\nabla\phi_v|\leq A/\epsilon$$
where $A=A(\kappa_M)$ is a constant which depends only on the lower bounds on Ricci curvature.
\end{lemma}
\begin{proof}
 We will construct a partition of unity $\{\phi_v\}$ subordinated to the open cover $\{B(v,4\epsilon)\}_{v\in V}$. Choose $4 \epsilon < r_0<\mbox{inj}(M)$. Note that with this choice, the covering is uniformly locally finite. Indeed, it follows from the theorem of Bishop-Gromov that a point is contained in a finite number of balls of the covering, and because $\epsilon< 1$, this number is bounded above by a constant $A_1(\kappa)$ depending only on the lower bound $\kappa$ of the Ricci curvature.


Let $\chi: \mathbb{R}_+ \rightarrow [0,1]$ be a smooth function with $\chi(t)=1$ if $t\le \frac{9}{10}$ and $\chi(t)= 0$ if $t \ge 1$. The family of functions $\{\psi_v\}_{v\in V}$ is defined by
$$
\psi_v(x)=\chi(\frac{3}{10 \epsilon}d(x,v)).
$$

Note that $\psi_v$ is of class $C^{\infty}$: for $x=v$, $d$ is not smooth, but $\psi_v$ is constant, and around the cut locus and after, where again $d$ may not be smooth, $\psi_v$ is constant, equal to $0$. The function $\psi_v$ is equal to $1$ on $B(v,3 \epsilon)$ and to $0$ outside of the ball $B(v,\frac{10\epsilon}{3})$.

\smallskip
As partition of unity, we choose $\{\phi_v\}_{v\in V}$, with
$$
\phi_v(x)= \frac{\psi_v(x)}{\sum_{w\in V} \psi_w(x)}.
$$

\noindent
\textbf{Control of the derivative of $\phi_v$.} We have
$$
\nabla \phi_v(x)= \frac{\nabla \psi_v(x) \sum_{w\in V} \psi_w(x)-\psi_v(x) \sum_{w\in V} \nabla \psi_w(x)}{(\sum_{w\in V} \psi_w(x))^2}
$$
with
$$
\vert \nabla \psi_v(x) \vert= (\max \vert \chi' \vert) \frac{3}{10 \epsilon}\vert \nabla d \vert \le \frac{3 A_2}{10 \epsilon}
$$
where $A_2=\max \vert \chi' \vert$.
As $\sum_{w\in V} \psi_w(x))^2 \ge 1$, it follows that
$$
\vert \nabla \phi_v(x) \vert \le \vert \nabla \psi_v(x) \vert \sum_{w\in V} \psi_w(x) +\psi_v(x) \sum_{w\in V} \vert \nabla \psi_w(x)\vert \le \frac{6A_1(\kappa)A_2}{10\epsilon}.
$$

\end{proof}

The \emph{smoothing} $F=S^Mf\in C^\infty(M)$ of
a function $f:V\rightarrow\mathbb{R}$ is defined by
$$F=\sum_{v\in V}f(v)\phi_v.$$ 

The following two Lemmas are very simple, but fundamental since they will allow the use of some arguments from \cite{cha2}.
\begin{lemma}\label{lemma:commute}
The functions $\phi_v\mid_\Sigma$ with $v\in V_\Sigma$ form a partition of unity on the boundary $\Sigma$. This allows the definition of the smoothing operator
  $$S^\Sigma:\ell^2(V_\Sigma)\rightarrow C^\infty(\Sigma),$$
  which commutes with restriction:
  \begin{gather}\label{equation:smoothingrestrictioncommute}
    S^\Sigma\left(f\Bigl|\Bigr._{V_\Sigma}\right)=(Sf)\Bigl|\Bigr._{\Sigma},\qquad\forall f\in\ell^2(V).
  \end{gather}
\end{lemma}
\begin{proof}
  This follows directly from the fact that a function $\phi_v$ for which $v\in V_I$ is supported away from $\Sigma$. 
\end{proof}

\begin{lemma}\label{lemma:SDvsDS} Let $f$ be a function on $V$. If $v \in V_{\Sigma}$, we have
$$
(\mathcal D \mathcal S f)(v)=(\mathcal D^{\Sigma} \mathcal S^{\Sigma} f)(v)
$$
where $\mathcal D^{\Sigma}$ and $\mathcal S^{\Sigma} $ denote the discretization and the smoothing on $\Sigma$ with the induced Riemannian metric.

\medskip
If $F$ is a differentiable function on $M$ and $x \in \Sigma$, we have
$$
(\mathcal S \mathcal D F)(x)=(\mathcal S^{\Sigma} \mathcal D^{\Sigma} F)(x)
$$
\end{lemma}

\begin{proof}
The proof is a direct consequence of the definition of the discretization we associate to $M$ and of the way to discretize and to smooth.

The restriction to $\Sigma$ of the smoothing of $f$ takes only account of the values of $f$ at point of $V_{\Sigma}$ and the restriction of the partition of unity we have defined is a partition of unity on $\Sigma$. So, for $x \in \Sigma$, we have $(\mathcal S f)(x)=(\mathcal S^{\Sigma} f)(x)$.

By definition, the same is true for the discretization because in order to discretize a function $F$ at points of $V_{\Sigma}$, we take the mean of $F$ restricted to balls of $\Sigma$. So, for $v\in V_{\Sigma}$, we have $(\mathcal D F)(v)=(\mathcal D^{\Sigma} F)(v)$.

From these two facts, we deduce immediatly the lemma.
\end{proof}

\begin{lemma}\label{lemma:basicsmoothing}
There exist constants $\Cl{basicsmoothing1}, \Cl{basicsmoothing2},\Cl{basicsmoothing3}$ depending only on $\kappa,r_0,n$ such that any $f\in\ell^2(V)$ satisfies
  \begin{gather*}
    \|Sf\|_M^2\leq \Cr{basicsmoothing1}\|f\|_V^2,\\\nonumber
    \|\nabla Sf\|_M^2\leq\Cr{basicsmoothing2}q(f),\\\nonumber
	\|Sf\|_{\Sigma}^2\leq \Cr{basicsmoothing3}\|f\|_{V_\Sigma}^2.\nonumber
  \end{gather*}
\end{lemma}
\begin{proof}
   The proof of the two first inequalities follows exactly the same arguments as that of paragraph $(\mathcal{S}:ii)$
  and $(\mathcal{S}:iii)$ in~\cite[section VI.5.2]{cha2}.
	
	\smallskip
	For the last inequality, we use the observation of Lemma \ref{lemma:commute}: using again paragraph $(\mathcal{S}:ii)$
  and $(\mathcal{S}:iii)$ in~\cite[section VI.5.2]{cha2}, we get that
	$$
	\|S^{\Sigma}(f\Bigl|\Bigr._{V_\Sigma})\|_{\Sigma}\leq \Cl{basicsmoothing4}\|f\|_{V_{\Sigma}}.
	$$
	
	But we have
	$$
	S^{\Sigma}(f\Bigl|\Bigr._{V_\Sigma})=(Sf)\Bigl|\Bigr._{\Sigma}
	$$
	and
	$$
	\|fs\|_{V_{\Sigma}}=\|f\Bigl|\Bigr._{V_\Sigma}\|_{V}.
	$$
\end{proof}

\section{Proof of the main comparison inequality}\label{section:proofmaincomparison}
The proof of Theorem~\ref{theorem:MainSpectralComparison} is broken
down between Proposition \ref{prop:down} and Proposition \ref{prop:up}.

\begin{prop}\label{prop:down}
  There exists a constant $\Cl{halftruth}$ depending only on $\kappa,r_0, n$ such that 
  $$\sigma_k(\Gamma,V_\Sigma)\leq \Cr{halftruth}k\sigma_k(M).$$
\end{prop}

Let $F_1,\cdots,F_k$ be Steklov eigenfunctions on $M$. The space
$$E_k:=\mbox{span}(DF_1,\cdots,DF_k)$$ will be used in the variational
characterization of $\sigma_k(\Gamma,V_\Sigma)$. One first needs to
ensure that this space is $k$-dimensional. This is not true in
general, but will hold for  low energy functions thanks to the
following quantitative injectivity property.
\begin{lemma}\label{lemma:discretizationInjective}
  There is a constant $0<\Cl{treshold1}<1/4$ with the usual dependance such that
  $\sigma_k(M)k<\Cr{treshold1}$ implies
  \begin{align}
    \|DF\|_{V_\Sigma}\geq \frac{\Cr{basicsmoothing1}^{-1}}{2}\|F\|_\Sigma.
  \end{align}
\end{lemma}
\begin{proof}[Proof of Lemma \ref{lemma:discretizationInjective}]
  Observe that
  Lemma~\ref{lemma:basicsmoothing} and the triangle inequality implies
  \begin{align}\label{inequality:goingdown1}
    \|DF\|_{V_\Sigma}\geq \Cr{basicsmoothing1}^{-1}\|SDF\|_\Sigma\geq \Cr{basicsmoothing1}^{-1}(\|F\|_\Sigma-\|F-SDF\|_\Sigma).
  \end{align}
  The argument in \cite[p. 183]{cha2} shows that 
  $$\|F - SDF\|_\Sigma \leq \Cl{SDF}\|\nabla F\|_\Sigma,$$
  for some constant $\Cr{SDF}$.
  Together with inequality~(\ref{inequality:goingdown1}) and Corollary \ref{coro:linearcombiFourrierStuff} this leads to
  \begin{align}
    \|DF\|_{V_\Sigma}&\geq \Cr{basicsmoothing1}^{-1}(\|F\|_\Sigma-\Cr{SDF}\|\nabla^{\Sigma} F\|_\Sigma)\\
    &=\Cr{basicsmoothing1}^{-1}\|F\|_\Sigma\|(1-\sqrt{k}\Cr{SDF}\frac{\|\nabla^{\Sigma}F\|_M}{\|F\|_\Sigma})
    \geq
    \|F\|_\Sigma\Cr{basicsmoothing1}^{-1}\|(1-\Cr{SDF}\sqrt{k\sigma_k}).
  \end{align}
  One can therefore take $\Cr{treshold1}=\min((2\Cr{SDF})^{-2},1/4)$.
\end{proof}

Everything is now in place for the proof of Proposition~\ref{prop:down}
\begin{proof}[Proof of Proposition~\ref{prop:down}]
  In the situation that $\sigma_k(M)k<\Cr{treshold1}$, it follows from
Lemma~\ref{lemma:discretizationInjective} that the space $E_k$ is
$k$-dimensional and because $C_9<1/4$,
Lemma \ref{lemma:discretizationenergybound} leads to
\begin{align*}
  \sigma_k(\Gamma,V_\Sigma)\leq\frac{q(DF)}{\|DF\|_\Sigma^2}\leq
  4\Cr{discretizationenergy}\Cr{basicsmoothing1}^2\frac{\|\nabla F\|_M^2}{\|F\|_\Sigma^2}\leq 4\Cr{discretizationenergy}\Cr{basicsmoothing1}^2\sigma_k(M).
\end{align*}

On the other hand, if $\sigma_k(M)k\geq \Cr{treshold1}$, then one has
$$\sigma_k(\Gamma,V_\Sigma)\leq k\Cr{treshold1}^{-1}\sigma_k(M)\sigma_k(\Gamma,V_\Sigma).$$
Now, let $\nu$ be the maximal degree of a vertex $v\in V_\Sigma$. Then for
$k=|V_\Sigma|$, one has $\sigma_k(\Gamma) \leq \nu,$ so that
$$\sigma_k(\Gamma,V_\Sigma)\leq \Cr{treshold1}^{-1}k\nu\sigma_k(M).$$

It follows that
$$\sigma_k(\Gamma,V_\Sigma)\leq
(4\Cr{discretizationenergy}\Cr{basicsmoothing1}^2+\Cr{treshold1}^{-1}k\nu)\sigma_k(M)
\leq (4\Cr{discretizationenergy}\Cr{basicsmoothing1}^2+\Cr{treshold1}^{-1}\nu)k\sigma_k(M)$$
and one concludes by setting
$A=(4\Cr{discretizationenergy}\Cr{basicsmoothing1}^2+\Cr{treshold1}^{-1}\nu)$.
\end{proof}

We are know ready to move on to the second comparison inequality.
\begin{prop}\label{prop:up}
  There exists a constant $B$ such that the following holds for each $k\leq |V_\Sigma|$:
  $$\sigma_k(M)\leq B\sigma_k(\Gamma,V_\Sigma).$$
\end{prop}

Let $f_1,\cdots,f_k$ be Steklov eigenfunctions on the graph
$(\Gamma,V_\Sigma)$. The space 
$$E_k:=\mbox{span}(Sf_1,\cdots,Sf_k)$$
will be used in the variational characterization of $\sigma_k(M)$. 
As in the smooth case, we show that this space is $k$-dimensional.
\begin{lemma}\label{lemma:smoothingInjective}
  There exists constants $\Cl{treshold2},
  \Cl{constantsmoothinjective}>0$ such that
   $\sigma_k(\Gamma,V_\Sigma)<\Cr{treshold2}$ implies
   $$\|Sf\|_\Sigma^2\geq\Cr{constantsmoothinjective}\|f\|_{V_\Sigma}^2\quad\forall f\in E_k$$
\end{lemma}
\begin{proof}[Proof of Lemma \ref{lemma:smoothingInjective}]

 Let $f=a_1f_1+...+a_kf_k$ and $F=\mathcal S f = a_1\mathcal S f_1+...+a_k \mathcal S f_k \in E_k$. 

\medskip
We denote by $F^{\Sigma}$ and $f^{\Sigma}$ the restriction of $F$ to
$\Sigma$ and of $f$ to $V_{\Sigma}$, respectively. Because we are working in the
closed manifold $\Sigma$, we can use point (8) from the paper
\cite{mantuano}, which states in our situation that  
$$
\Vert \mathcal D^{\Sigma}F^{\Sigma} \Vert_{V_{\Sigma}} \le C\Vert F^{\Sigma} \Vert_{\Sigma},
$$
for some constant $C$ depending only on $\kappa,n,r_0$.
It follows that
$$
\Vert F^{\Sigma}\Vert_{\Sigma} \ge \frac{1}{C} \Vert \mathcal D^{\Sigma} F^{\Sigma} \Vert_{V_{\Sigma}} \ge
\frac{1}{C}\left( \Vert f^{\Sigma}\Vert_{V_{\Sigma}}-\Vert f-D^{\Sigma}S^{\Sigma} f^{\Sigma}\Vert_{V_\Sigma}\right).
$$
Working on $\Sigma$ and using point (10) of \cite{mantuano}:
$$
\Vert f-D^{\Sigma}S^{\Sigma} f^{\Sigma}\Vert_{V_\Sigma} \le C'q^\Sigma(f)\leq C'q(f),
$$
for another constant $C'$. Here $q^\Sigma(f)$ is the energy of $f$ restricted to the boundary $V_\Sigma\subset V$.
Because $f\in\mbox{Span}(f_1,\cdots,f_k)$, we have
$$
\frac{q(f) }{\Vert f^{\Sigma}\Vert_{V_{\Sigma}}^2} \le \sigma_k(\Gamma,V_\Sigma).
$$
Setting $\Cr{treshold2}=\frac{1}{2C'}$, observe that
the inequality $\sigma_k(\Gamma,V_\Sigma) \le \Cr{treshold2}$ implies
$\|F\|_\Sigma^2\geq \frac{1}{2C}\|f\|_{V_\Sigma}^2$. The proof is
complete, with $\Cr{constantsmoothinjective}=\frac{1}{2C}$.
\end{proof}

We are know ready to finish the proof of the comparison inequality.

\begin{proof}[Proof of Proposition \ref{prop:up}]
In the situation that $\sigma_k(\Gamma,V_\Gamma)<\Cr{treshold2}$, it
follows from Lemma \ref{lemma:smoothingInjective} that the space
$E_k=\mbox{span}(Sf_1,\cdots,Sf_k)$ is $k$-dimensional. In combination
with Lemma \ref{lemma:basicsmoothing}, this lead for each $F\in E_k$ to
\begin{gather}\label{inequalityProof1}
  \sigma_k(M)\leq\frac{\|\nabla F\|^2}{\|F\|_M^2}
  \leq\Cr{basicsmoothing2}\Cr{constantsmoothinjective}^{-1}\frac{q(f)}{\|f\|_{V_\Sigma}}
  \leq \Cr{basicsmoothing2}\Cr{constantsmoothinjective}^{-1}\sigma_k(\Gamma,V_\Sigma).
\end{gather}
On the other hand, if $\sigma_k(\Gamma,V_\Gamma)\geq\Cr{treshold2}$, then one has
$$\sigma_k(M)\leq \Cr{treshold2}^{-1}\sigma_k(\Gamma,V_\Sigma)\sigma_k(M).$$
The proof is completed by giving a rough upper bound on $\sigma_k(M)$.
Because $k\leq|V_\Sigma|$ and $V_\Sigma\subset\Sigma$ is
$\epsilon$-separated,  there exists $k$ disjoint balls
$B_1,\dots,B_k\subset \Sigma$ of radius $\epsilon$.   Let $f_j$ be a first eigenfunction corresponding to the
first Dirichlet eigenvalue $\lambda_1(B_j)$. Each function
$f_j$ is extended to a function $\phi_j:B_j\times[0,1]\subset
M\rightarrow\R$ defined on this set by
$$\phi_j(x,t)=f_j(x)(1-t),$$
and then extended by 0 elsewhere in $M$.
It follows that
$$\frac{\|\nabla\phi_j\|_{M}^2}{\|f_j\|_\Sigma^2}\leq
\frac{\|\nabla^\Sigma
  f_j\|_\Sigma^2+\|f_j\|_\Sigma^2}{\|f_j\|_\Sigma^2}=\lambda_1(B_j)+1.$$
Now, Cheng's theorem \cite{cheng} states that $\lambda_1(B_j)\leq C,$ for some constant
$C=C(n,\kappa,\epsilon)$. 
As the functions $\phi_j$ are compactly supported in the disjoint
domains $B_j\times[0,1]$, the min-max characterization of $\sigma_k$
implies that $\sigma_k(M)\leq C+1.$
Hence
$$\sigma_k(M)\leq (C+1)\Cr{treshold2}^{-1}\sigma_k(\Gamma,V_\Sigma).$$
Together with Inequality \eqref{inequalityProof1}, this implies that
$$\sigma_k(M)\leq B\sigma_k(\Gamma,V_\Sigma),$$
for $B=\max\{(C+1)\Cr{treshold2}^{-1}, \Cr{basicsmoothing2}\Cr{constantsmoothinjective}^{-1}\}$.
\end{proof}

\section{Applications}
\label{section:applications}

In this section, we present the three applications which were mentionned in the introduction. They are similar to each others. The general strategy is as follows:
\begin{itemize}
\item Construct a sequence of graphs $\{G_l\}_{l\in\N}$ with some desired spectral property;
\item Construct a sequence of surfaces $\{\Omega_l\}_{l\in\N}$ associated to $G_l$;
\item Obtain an $\epsilon$-discretization $(\Gamma_l,B_l)$ of $\Omega_l$;
\item Prove that $\Gamma_l$ and $G_l$ are roughly isometric;
\item Conclude using Theorem \ref{theorem:MainSpectralComparison} and Proposition \ref{proposition:spectrumRoughlyIsomGraphs}.
\end{itemize}

\subsection{Spectral stability under quasi-isometries}
In order to compare the Steklov spectrum of our manifolds with the discrete Steklov spectrum of a discretization, it was necessary to suppose that a neighborhood of the boundary $\Sigma$ is isometric to the product $\Sigma \times [0,1[$. This is a strong hypothesis, which however can be relaxed to having a quasi-isometry with the product $\Sigma \times [0,1[$, with uniform control on constants.
 
\begin{prop}\label{proposition:quasiisocontrol}
 Let $M^n$ be a compact manifold with smooth boundary $\Sigma$ and let $g_1,g_2$ be two Riemannian metrics on $M$. Suppose the existence of a constant $A\ge 1$ such that for each $x\in M$ and $0\neq v \in T_xM$ we have
 $$
 \frac{1}{A} \le \frac{g_1(x)(v,v)}{g_2(x)(v,v)} \le A.
 $$
 Then the Steklov spectrum with respect to $g_1$ and $g_2$ satisfies the
 $$
 \frac{1}{A^{2n+1}}\le \frac{\sigma_k(M,g_1)}{\sigma_k(M,g_2)} \le A^{2n+1},
 $$
 and for the normalized Steklov spectrum, we have
 $$
 \frac{1}{A^{2n+2}}\le \frac{\tilde{\sigma}_k(M,g_1)}{\tilde{\sigma}_k(M,g_2)} \le A^{2n+2}.
 $$
\end{prop}
\begin{proof}
This follows directly from the variational characterization of the
Steklov eigenvalues $\sigma_k$. The earliest paper where this
principle was used extensively is that of Dodziuk \cite{dodziuk},
where the spectrum of the Laplace-Beltrami operator acting on forms is
studied. 
\end{proof}

\subsection{Planar domains with large eigenvalues}

The proof of Theorem \ref{theorem:planardomains} is based on gluing copies of three different planar building blocks.

\begin{defi}\label{defi:planfunpiece}
A \emph{planar fundamental piece} is a domain $D\subset\mathbb R^2$
bounded by smooth successive arcs
$\gamma_1,\Gamma_1,...,\gamma_{n-1},\Gamma_{n-1},\gamma_n=\gamma_1$
meeting orthogonaly (see Figure \ref{fig:flatcross}) such that
the arc $\Gamma_i$ is a straight segment of length $1$ and in a
neighbourhood of each $\Gamma_i$, $D$ is isometric to $\Gamma_i\times
[0,1]$ (a square of side $1$). 
\end{defi}

\begin{figure}[ht]
  \centering
  \includegraphics[width=6cm]{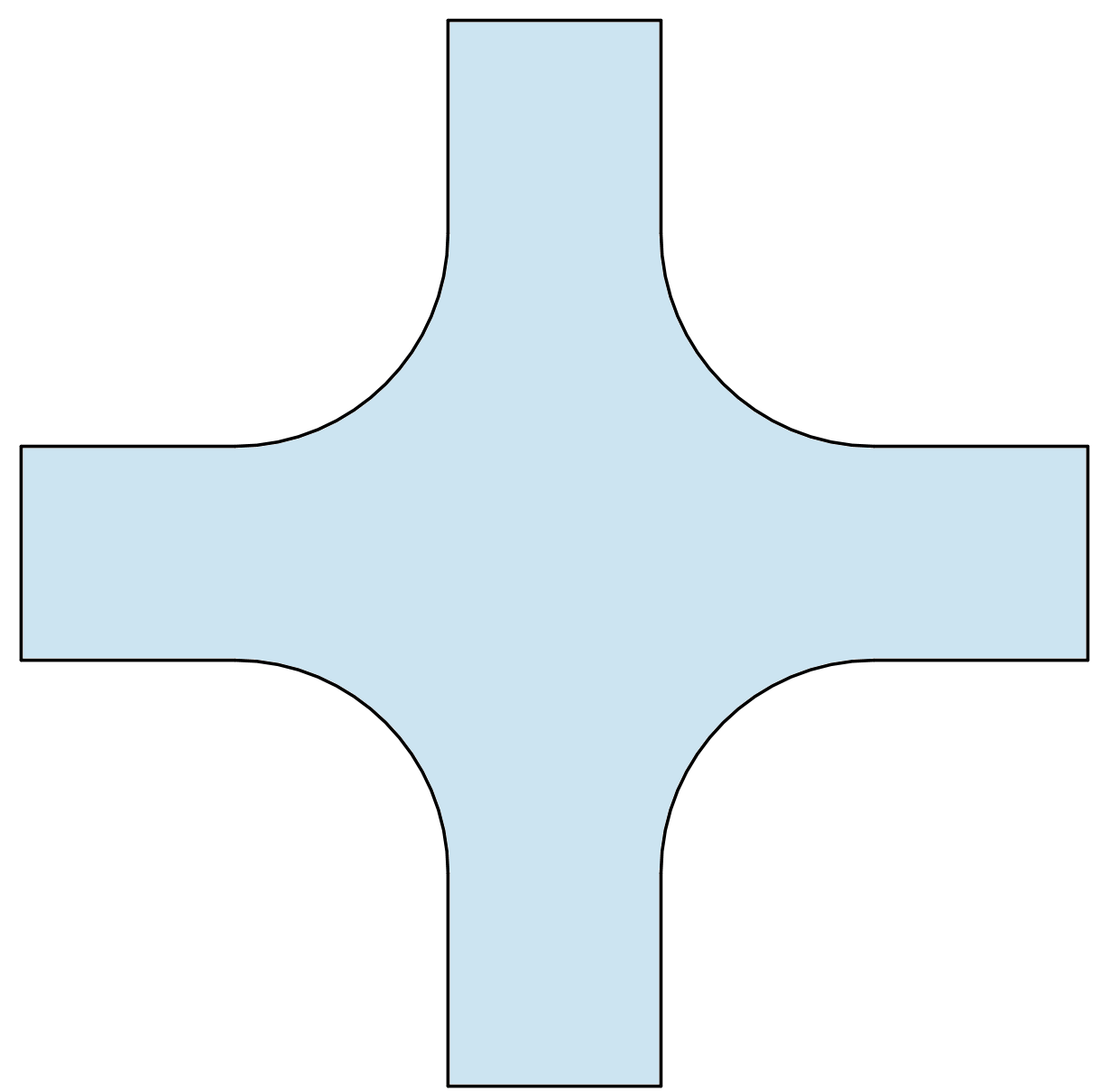}
  \caption{A flat cross domain.}
  \label{fig:flatcross}
\end{figure}

Let us denote by $N(\delta)$ the $\delta$-neighbourhood of $\gamma:=\cup_i\gamma_i$, that is 
$$N(\delta)=\{x\in D: d(x,\gamma)\le \delta\}.$$
\begin{lemma}\label{lemma:quasiproductfundpiece}
  Let $D\subset\R^2$ be a flat fundamental piece.
 Let $g_0$ be the euclidean metric. There exist numbers $K,\delta>0$ and a Riemannian metric $g$ on $D$ such that
\begin{itemize}
\item $(N(\delta),g)$ is isometric to $\gamma\times [0,1]$;
\item The Riemannian metrics $g_0$ and $g$ are quasi-isometric  with constant $K$;
\item The Riemannian metrics $g$ and $g_0$ are homothetic on $D\setminus N(3\delta)$ and on the square ends of $D$.
\end{itemize}

\end{lemma}
\begin{proof}
Let $s$ be the arclength parameter along $\gamma$. Using the distance
$t$ to $\gamma$ as a second parameter leads to the Fermi parallel
coordinates, which are defined in a neighborhood $\mathcal{O}\subset
D$ of $\gamma$. In this coordinate system, the eudlidean metric is
expressed by
$$g_0(s,t)=\phi(s,t)ds^2+dt^2,$$
where the smooth function $\phi$ satisfy $\phi(s,0)=1$ (Gauss
Lemma). Let $\delta>0$ be such that the restriction of $\phi$ to
$N(3\delta)$ is smaller than 2. On this neighborhood the euclidean
metric $g_0$ is quasi-isometric with ratio 2 to the product metric
$g'$, which in Fermi coordinates is expressed by $g'(s,t)=ds^2+dt^2$.  
Let $\chi:[0,3\delta] \rightarrow [0,1]$ be a smooth increasing function taking the value $0$ on $[0,\delta]$ and the value $1$ on $[2\delta,3\delta]$. Using the Fermi coordinates again, define a new metric on $N(3\delta)$ by
$$g_\delta(s,t)=\chi(t) g_0(s,t)+(1-\chi(t))g'(s,t).$$
On $N(\delta)$ this metric coincide with the product metric $g'$, while on 
$N(3\delta)\setminus N(2\delta)$ it coincides with the euclidean metric $g_0$. It can therefore be extented to a metric (still denoted $g_\delta$) which is defined on the full domain $D$.

Note that on the square ends of the fundamental piece $D$, one has $g_\delta=g'=g_0$. 	
In order to obtain a cylindrical boundary of length one, define the metric
$$g=\frac{1}{\delta^2}g_\delta.$$
This metric satisfy all the required condition.
Indeed it is $2$ quasi-isometric to $\frac{1}{\delta^2} g_0$ by construction, and equal to
$\frac{1}{\delta^2} g_0$ on $D\setminus N(3\delta)$ and on the square ends of $D$.
\end{proof}

Let $D_1,\cdots,D_m$ be planar fundamental pieces. Let $\Omega\subset\R^2$ be a planar domain of the form
$$\Omega=\mbox{interior }\overline{\dot{\bigcup}_{i=1}^n\Omega_i}$$ 
where each $\Omega_i$ is one of the fundamental piece, with the external boundary $\gamma$ of each piece included in the boundary $\partial\Omega$. We call such a domain a \emph{puzzle domain} (See Figure \ref{figure:puzzle}).

\begin{cor}\label{coro:puzzlequasi}
 Let $D_1,\cdots,D_m$ be planar fundamental pieces. There exist
 constant $K,\kappa,r_0>0$ such that any puzzle domain $\Omega$ based
 on $D_1,\cdots,D_m$ is $K$-quasi-isometric to a Riemannian surface
 $(\Omega',g)$ in the class $\mathcal{M}(\kappa,r_0,2)$. 
\end{cor}
\begin{proof}
  This follows from Lemma \ref{lemma:quasiproductfundpiece} and the
      fact that a finite number of fundamental pieces are used.
\end{proof}

\begin{remark}
One way to think of $\Omega_l$ intuitively is that it is a
"thickening" of the graph $G_l$ (perceived as a subset of the usual
lattice $\Z^2\subset\R^2$), that is by considering a tubular
neighbourhood of the corresponding lattice graph. Nevertheless, it is
easier to analyse the situation precisely when we build $\Omega_l$ by
gluing fundamental pieces.
\end{remark}

We are now ready to proceed with the proof of our first application.
\begin{proof}[Proof of Theorem \ref{theorem:planardomains}.]
For each $l\in\N$, consider the graph $G_l=(V_l,E_l)$ defined through the usual lattice $\Z^2\subset\R^2$ as follows:
$$
V_l= \{(a,b) \in \Z^2 \subset \R^2: \ 0\le a,b \le l \}
$$
and declare the vertices $(a_1,b_1)$ and $ (a_2,b_2)$ to be adjacent if they are related in the lattice $\mathbb Z^2$.
In order to obtain a graph with boundary, a distinguished subset of
vertices has to be chosen. For the present application, chose
$B_l=V_l$.  In this very special situation where the boundary coincide
with the full graph, the spectrum of the discrete Steklov problem on
coincide with the spectrum of the combinatorial Laplacian of the graph
(See Definition \ref{defi:steklovgraph} and Remark
\ref{remark:fullvertices}).  

\begin{figure}
  \includegraphics[width=6cm]{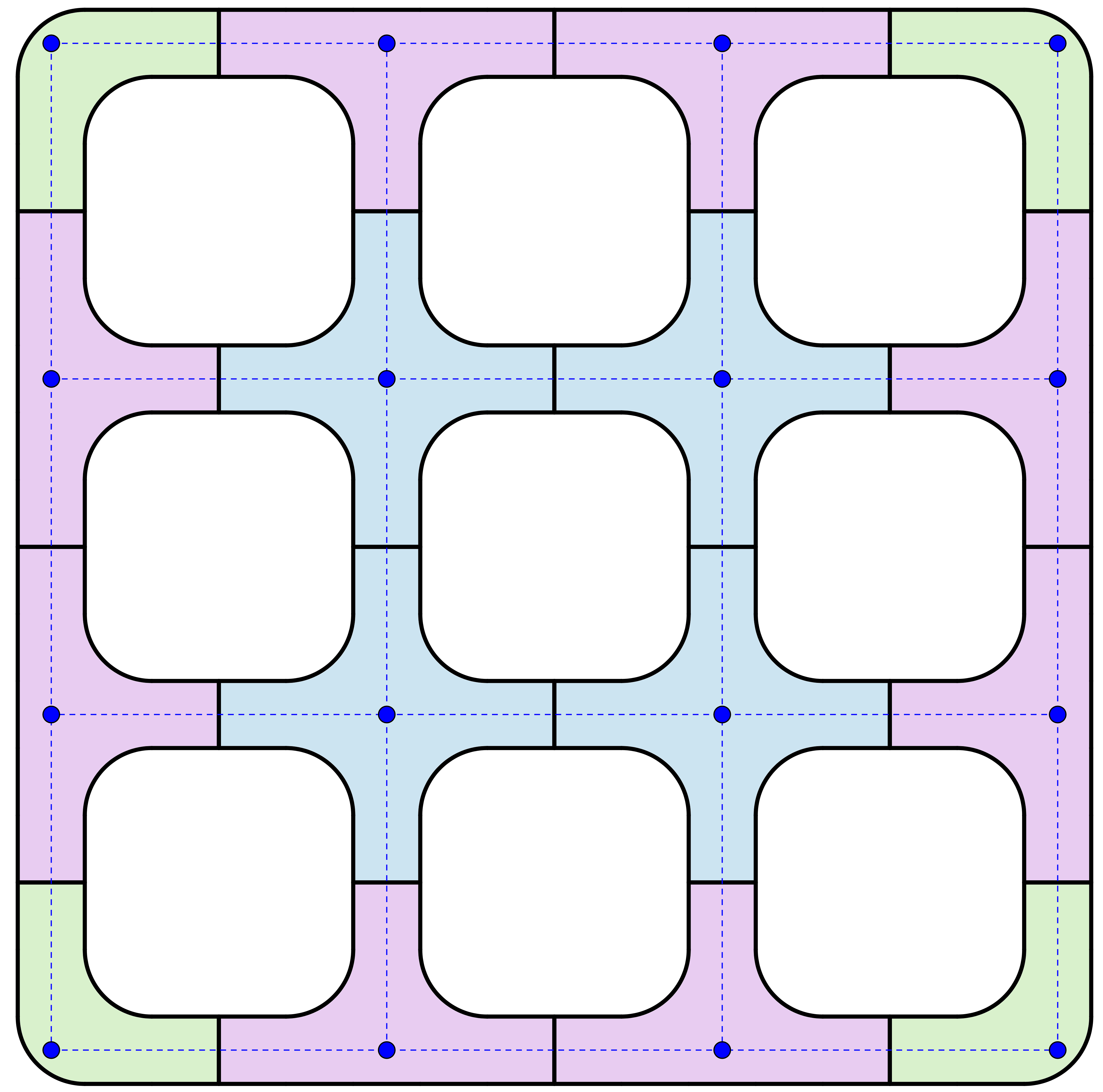}
  \caption{A puzzle domain with three fundamental pieces.}
  \label{figure:puzzle}
\end{figure}

Vertices of $G_l$ are either of degree 2,3 or 4. To each of these
degree, their correspond one planar fundamental piece: $D_2, D_3$ and
$D_4$, as illustrated in Figure \ref{figure:puzzle}. For each $l\in\N$, these are used to construct a puzzle domain $\Omega_l\subset\R^2$ as follows: 
\begin{itemize}
\item To each vertex $v\in V_l$, a congruent copy $\Omega_v$ of the fundamental piece $D_{\mbox{deg}v}$ is attached; \item If $v,w\in V$ are adjacent in the graph $G_l=(V_l,E_l)$ then $\Omega_v$ and $\Omega_w$ are placed so as to touch along a common straight $\Gamma$-segment of their boundary. 
\end{itemize}
It follows from Corollary \ref{coro:puzzlequasi} that there
exists a constant $K>0$ independant of $l\in\N$ such that each domain
$\Omega_l\subset\R^2$ is quasi-isometric (in the Riemannian sense)
with constant $K$ to a Riemannian surface $(M_l,g_l)$ in the class
$\mathcal{M}(\kappa,r_0,2)$ for some constants $\kappa,r_0>0$ which
are also independant of $l$. Moreover, it is clear from the
construction that $(M_l,g)$ is roughly isometric to the graph $G_l$,
with uniform constants $a,b,\tau$.

Let $\epsilon\in (0,r_0/4)$ and consider an $\epsilon$-discretization
$(\Gamma(M_l),V_{\Sigma}(M_l))$  of the surface $M_l$. This graph is
roughly-isometric to $M_l$ by
Corollary\ref{lemma:discretizationsRoughlyIso}.

\medskip
We are now in position to conclude.
It follows from Proposition \ref{proposition:spectrumRoughlyIsomGraphs} that there exists a constant $A_1\ge 1$ not depending on $l$ such that
$$
\frac{1}{A_1} \le \frac{\sigma_2(G_l)}{\sigma_2(\Gamma_l)} \le A_1.
$$
From our main comparison result (Theorem \ref{theorem:MainSpectralComparison}) we get a constant $A_2 \ge 1$ not depending on $l$ such that
$$
\frac{1}{A_2} \le \frac{\sigma_2(\Omega_l)}{\sigma_2(\Gamma_l)} \le A_2.
$$
As for the graph $G_l$ we have $B_l=V_l$, we have
$\sigma_2(G_l)=\lambda_2(G_l)$. It is well known that the first
non-zero eigenvalue of the the discrete Laplacian on this so called
lattice graph is $8\sin^2(\pi/l)$ (see \cite{mohar}). In particular, this implies the existence of a constant $A_3$ such that
$$
\sigma_2(G_l)=\lambda_2(G_l)\ge \frac{A_3}{l^2}
$$
The conclusion is that there exists a constant $C$ depending on $A,A_1,A_2,A_3$ but not depending on $l$ such that
$$
\sigma_2 (\Omega_l) \ge \frac{C}{l^2}.
$$

Now it is obvious by construction that there exists $B_1,B_2>0$ with
$$B_1 l^2 \le\mbox{Vol}(\Omega_l) \le B_2 l^2,\qquad B_1l^2 \le L(\Sigma_l)\le B_2 l^2.$$
This implies first that
$$
\sigma_2 (\Omega_l) L(\Sigma_l) \ge B_1C.
$$
and that
$$
I(\Omega_l) = \frac{L(\Sigma_l)}{\mbox{Vol}(\Omega_l)^{1/2}}\ge \frac{B_1}{\sqrt B_2} l.
$$
So that $I(\Omega_l) \to \infty$ as $l\to \infty$ and $\sigma_2 (\Omega_l) L(\Sigma_l)\ge B_1C$.
\end{proof}

\subsection{Flat surfaces}
The proof of the second application is similar.

\begin{proof}[Proof of Theorem \ref{theorem:LargeLambdaOneFlat}]

It follows from the classical probabilistic method that there exist
a an expanding family $\{G_l\}_{l\in\N}$ of $4$-regular graphs such
that the number of vertices $|V(G_l)|=l$ (See~\cite{pinsker}).
In particular, 
$\lim_{l\rightarrow\infty}|V(G_l)|=+\infty$ and 
$\lambda_1(G_l)$ is uniformly bounded below by a positive
constant $c$.  In order to obtain a surface $\Omega_l$ from the graph
$G_l$, a "flat cross domain" is used (see Figure
\ref{fig:flatcross}). This domain is a fundamental planar piece in the
sense of Definition \ref{defi:planfunpiece}.The construction of
$\Omega_l$ is similar to that used for the previous application, but
it is simpler since only one building block is required: a flat cross
$\Omega_v$ is associated to each vertex $v\in V(G_l)$, and the graph
structure is used to glue them together. It follows exactly as in
Corollary \ref{coro:puzzlequasi} that the surface $\Omega_l$ is
$K$-quasi-isometric to a surface $M_l$ in some class
$\mathcal{M}(\kappa,r_0,2)$ for some constants $K,\kappa,r_0$ which
are independant of $l\in\N$.

Given $\epsilon\in (0,r_0/4)$, let $(\Gamma_{M_l},V_{\Sigma_l})$ be an $\epsilon$-discretization of $M_l$. Lemma \ref{lemma:discretizationsRoughlyIso} says that $M_l$ is roughly-isometric to $(\Gamma_{M_l},V_{\Sigma_l})$ for some constants $a,b,\tau$ independant of $l\in\N$.

The end of the proof is exactly as that of the previous Theorem \ref{theorem:planardomains}.
\end{proof}

\subsection{Surfaces with large eigenvalues and connected boundary}
The basic idea behind the proof of Theorem \ref{theorem:LargeLambdaOneConnectedBdr} is similar to the two previous examples and to the construction of \cite{cg1}, but some care has to be taken to garantee the presence of only one boundary component. 
Intuitively, we proceed as follows: consider a family $G_l$ of expander graphs of degree 4 and genus $N+1$ as in the previous paragraph, and use it to construct a sequence $S_l$ of closed surfaces, using exactly one building block of the form given in Figure \ref{fig:fundamentalpiece}.
\begin{figure}[ht]
  \centering
  \includegraphics[width=4cm]{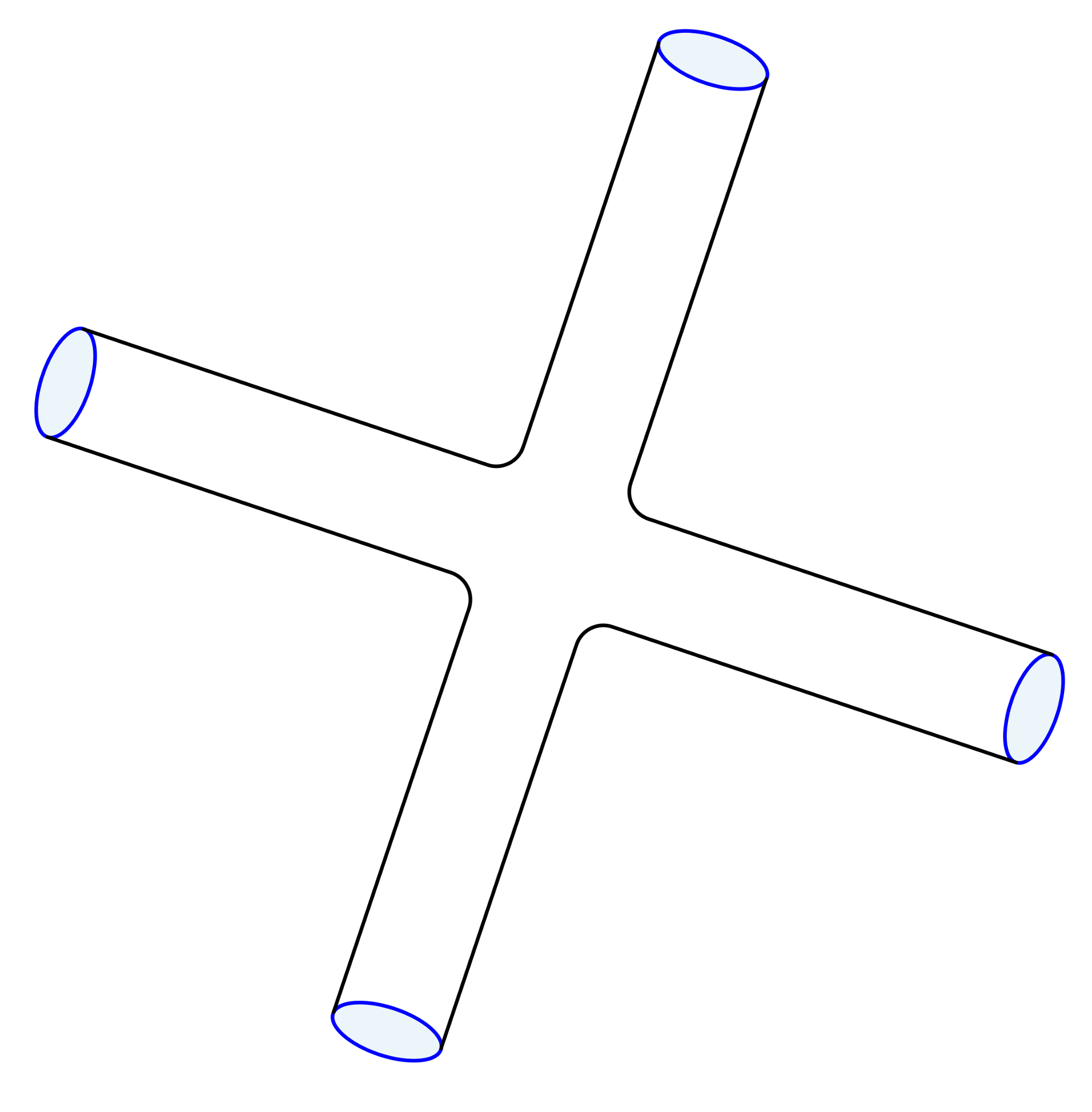}
  \caption{The fundamental piece $M_0$ for $k=4$}
  \label{fig:fundamentalpiece}
\end{figure}
Consider a \emph{maximal tree} on the graph $G_l=(V_l,E_l)$. One would like to embedd this tree in the surface $S_l$ and remove a neighborhood from $O_l$ the closed surface $S_l$. Because this neighborhood is a topological ball in $S_l$, the surface $\Omega_l:=S_l\setminus O_l$ has exactly one boundary component $\Sigma_l$, which is spread out in the surface: it visits each and every fundamental piece that was used to construct $S_l$, and its length is proportional to $l$. Nevertheless, it is not easy to embedd the graph $G_l$ in the closed surface $S_l$ while controlling a quasi-isometry class for $O_l$.  To get around this diffulty, 
4 different fundamental pieces will be used (See Figure \ref{fig:carvedpieces}). These fundamental pieces already carry part of the boundary of $\Omega_l$, which is built using the maximal tree and the five pieces which correspond to the degree of each vertex in the tree. What we have gained through this construction is that each fundamental piece has a fixed geometry, and this implies the existence of a quasi-isometry from a neighborhood of the boundary $\Sigma_l$ to $\Sigma_l\times [0,1)$ with constant not depending on $N$.
The rest of the proof is essentialy the same as those of the two previous examples.

\begin{proof}[Proof of Theorem~\ref{theorem:LargeLambdaOneConnectedBdr}]
Let $M_0$ be the smooth surface of genus $0$ with $4$ boundary
components illustrated in Figure~\ref{fig:fundamentalpiece}.
Each boundary component $B_1,\cdots, B_4$ has a neighbourhood
which is isometric to the cylinder $[0,1]\times S^1$, with 
boundary corresponding to $\{0\}\times S^1$. Moreover, the surface
$M_0$ is symmetric with respect to rotations of $90$ degrees.
From $M_0$ we build 4 new surfaces by carving out smooth curves as
illustrated in figure \ref{fig:carvedpieces}. 
\begin{figure}[ht]
  \centering
  \includegraphics[width=4cm]{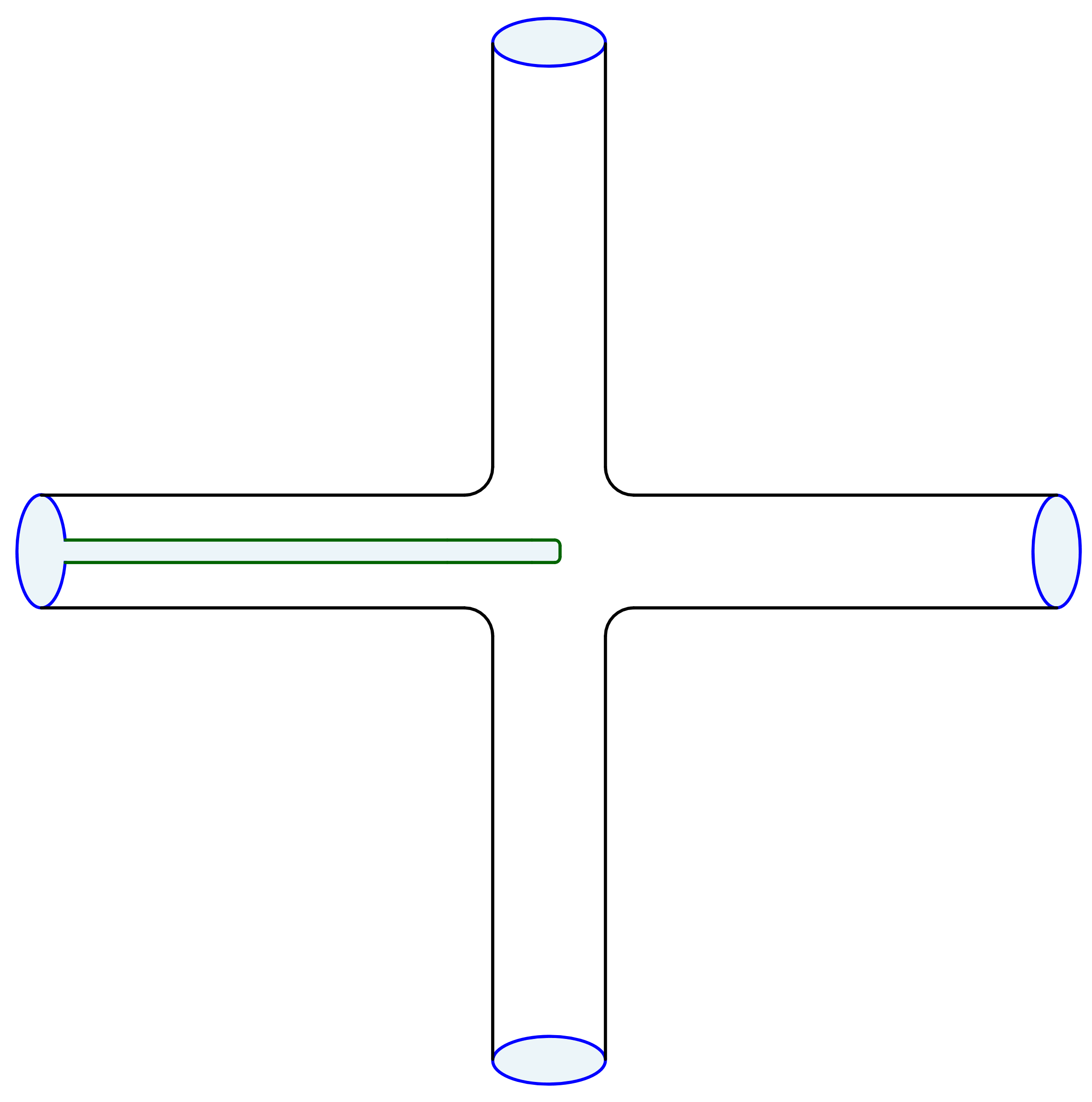}
  \includegraphics[width=4cm]{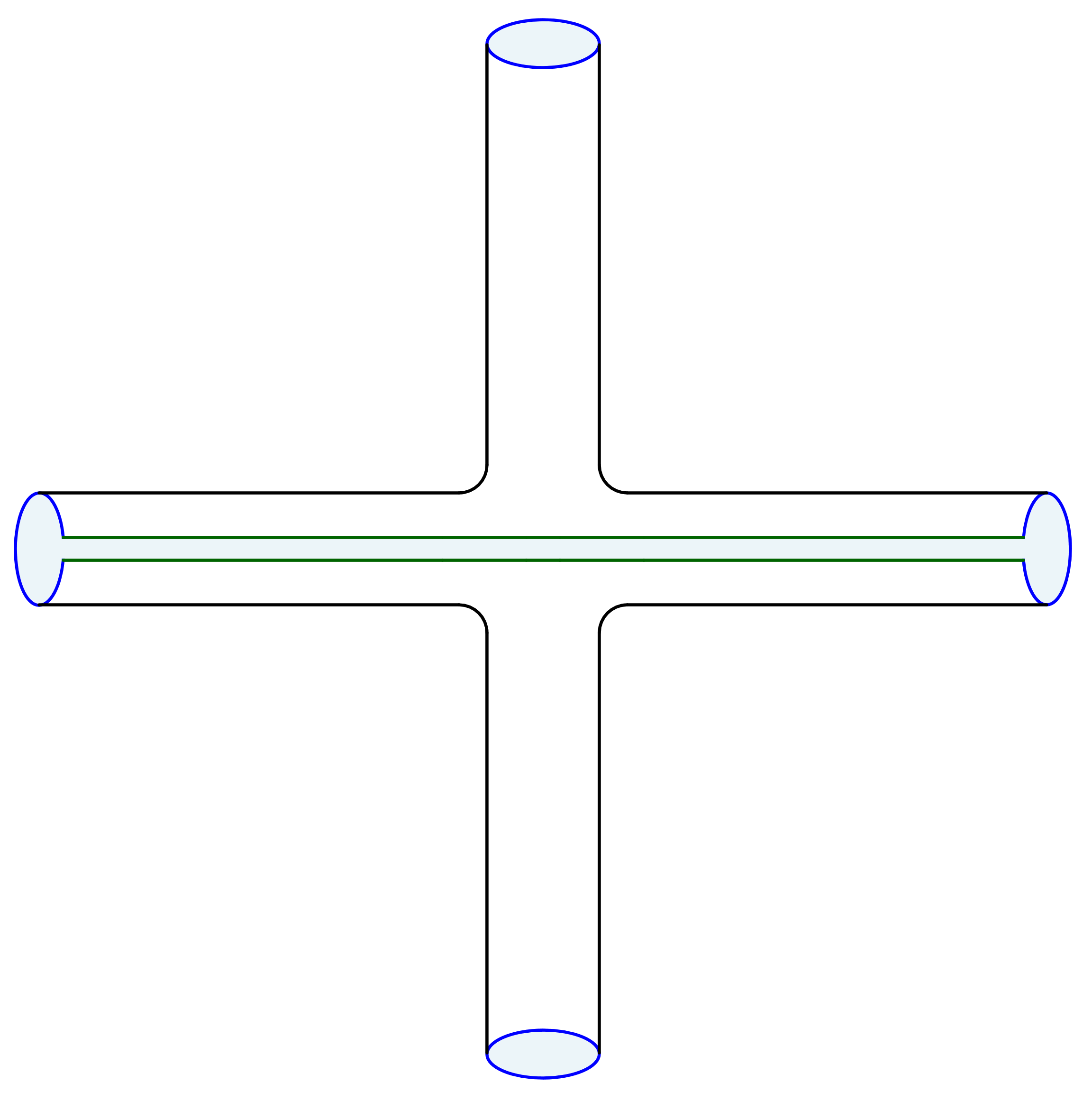}
  \includegraphics[width=4cm]{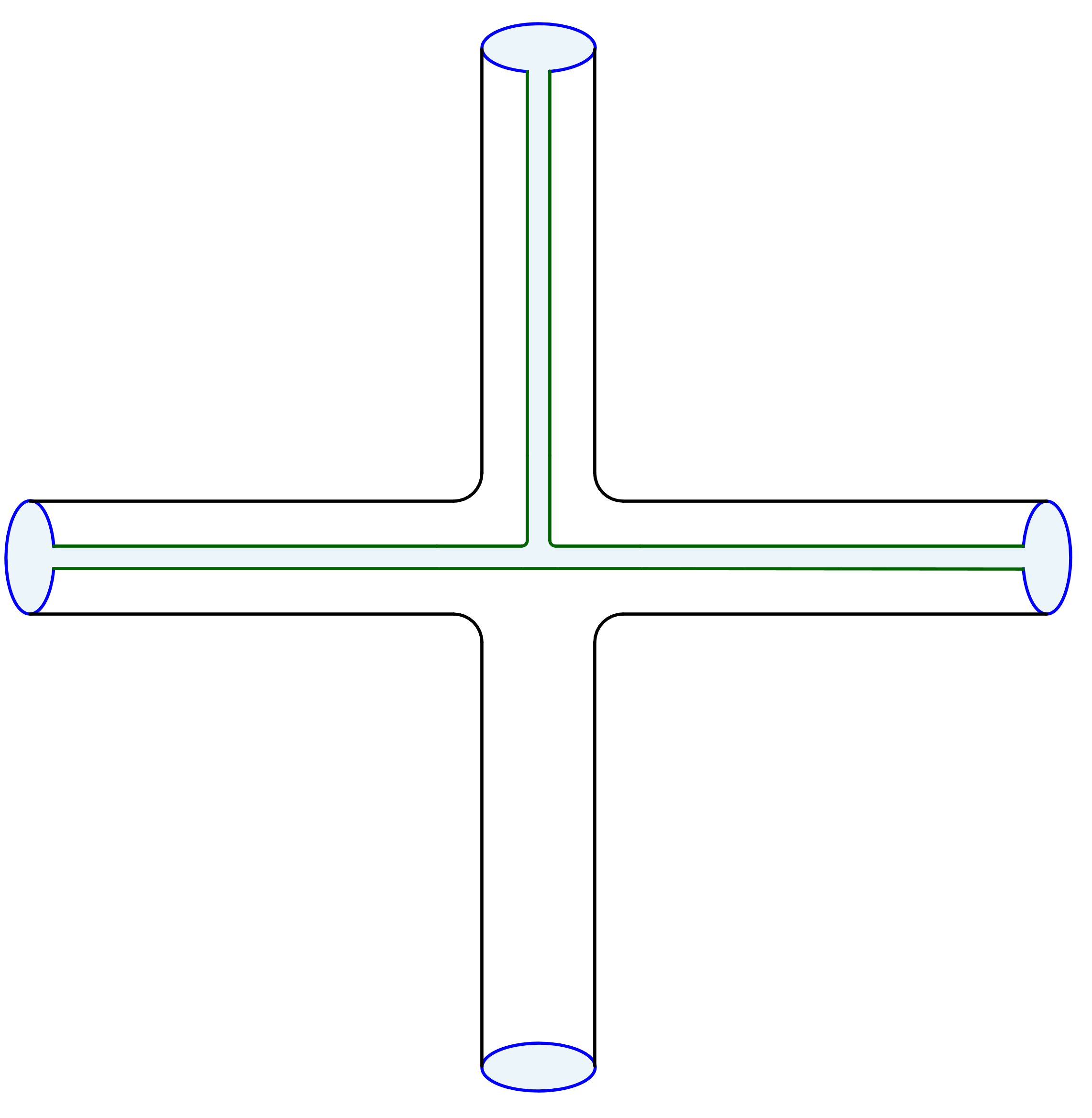}
  \includegraphics[width=4cm]{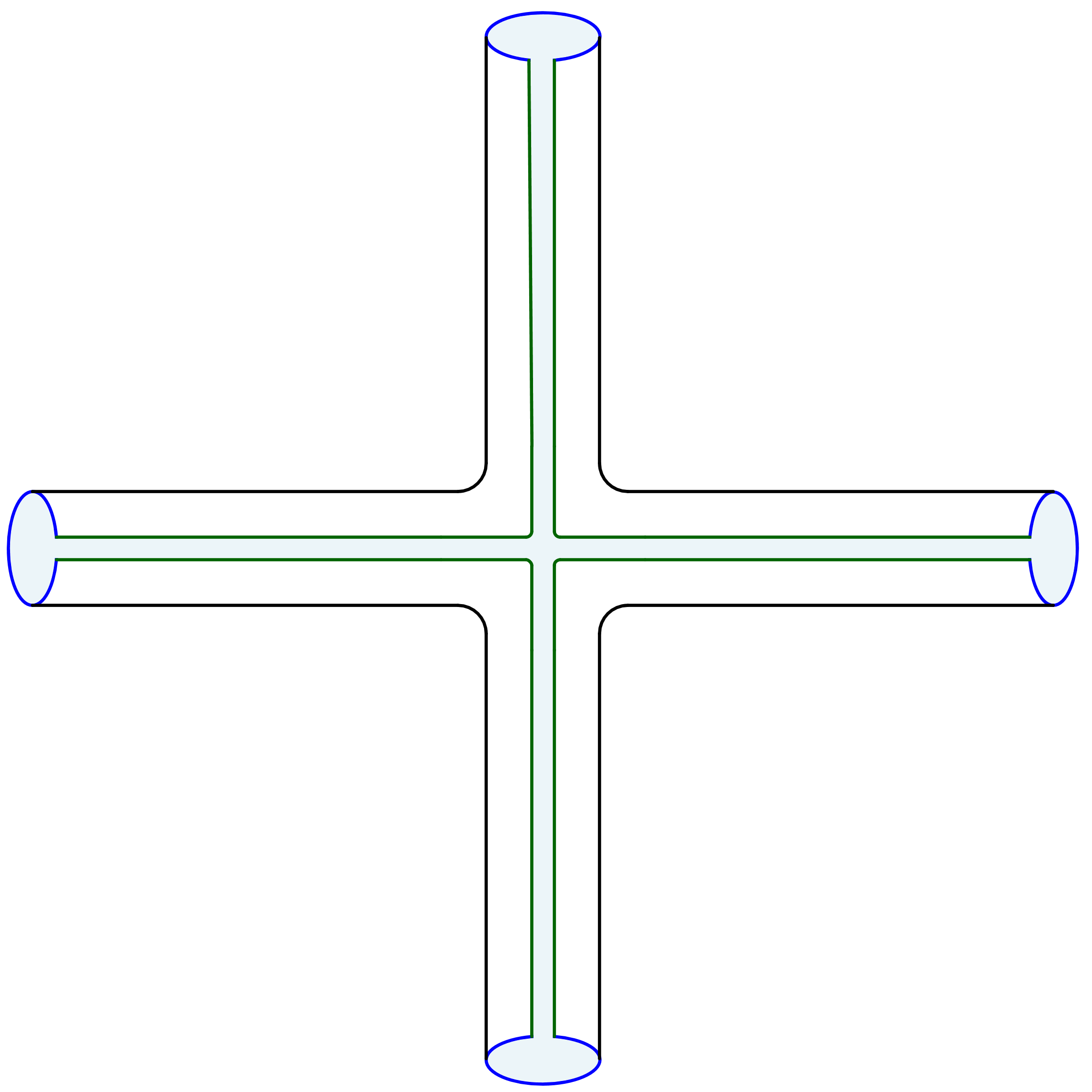}
  \caption{Carved fundamental pieces}
  \label{fig:carvedpieces}
\end{figure}
These surfaces are such that they match together seamlessly, so that they
can be used like the basic building blocks of a puzzle.
Let $G=(V,E)$ be a  finite connected regular graph of degree $4$.
Let $T$ be a maximal tree in $G$. 
To each vertex $v\in V$, a copy $M_v$ of one of the four building blocks is associated,
according to the degree of $v$ in the maximal tree $T$. 
These blocks are then glued together to obtain a surface
$\Omega_{G,T}$, with one boundary component
$\Sigma_{G,T}$. 
It is clear from this construction that each $\Omega_{G,T}$ is $K$-quasi-isometric to a surface $M_{G,T}$ in the class $\mathcal{M}(\kappa,r_0,2)$ for some constants $K,\kappa,r_0$ which are independant of the graph $G$ and maximal tree $T$. See Lemma \ref{lemma:quasiproductfundpiece} for details in the planar case.

Exactly as in the previous two applications, it follows from Proposition \ref{lemma:discretizationsRoughlyIso} that any $\epsilon$-discretization $(\Gamma_{M_G},V_{\Sigma_{G,T}})$ of $M_{G,T}$ is roughly isometric to $M_{G,T}$, with constants $a,b,\tau$ independent of the graph $G$ and of the maximal tree $T$. Moreover, the discretization is roughly isometric to the original graph $(G,T)$. 

The rest of the proof is also exactly as that of the previous Theorem \ref{theorem:LargeLambdaOneFlat}, using an expander sequence of graphs $G_l$ and applying the spectral comparison Proposition \ref{proposition:spectrumRoughlyIsomGraphs} and Theorem \ref{theorem:MainSpectralComparison}.
\end{proof}

\begin{remark}
 In all construction presented above, a choice is involved when we glue various building blocks together: we did not specify which boundary component of one fundamental piece is glued to the other. In the first construction (planar domains), there was no ambiguity, since we are using congruent copies of the fundamental pieces in the plane. For the last two applications, the choice involved does not affect the end result, despite the surface not being uniquely defined by the procedure. One way to resolve this non-uniqueness is to label the edges emanating from each vertex and label the gluing boundaries. See \cite{cg1} for more details.
 \end{remark}
 
\bibliographystyle{plain}
\bibliography{biblioBCG}

\def\cprime{$'$} \def\cprime{$'$} \def\cprime{$'$}
\begin{thebibliography}{10}

\bibitem{brooks1}
R.~Brooks.
\newblock The first eigenvalue in a tower of coverings.
\newblock {\em Bull. Amer. Math. Soc. (N.S.)}, 13(2):137--140, 1985.

\bibitem{buser1}
P.~Buser.
\newblock On the bipartition of graphs.
\newblock {\em Discrete Appl. Math.}, 9(1):105--109, 1984.

\bibitem{cha2}
I.~Chavel.
\newblock {\em Isoperimetric inequalities}, volume 145 of {\em Cambridge Tracts
  in Mathematics}.
\newblock Cambridge University Press, Cambridge, 2001.
\newblock Differential geometric and analytic perspectives.

\bibitem{cheng}
S.~Y. Cheng.
\newblock Eigenvalue comparison theorems and its geometric applications.
\newblock {\em Math. Z.}, 143(3):289--297, 1975.

\bibitem{ceg3}
B.~Colbois, A.~El~Soufi, and A.~Girouard.
\newblock Comparison.
\newblock {\em in preparation}.

\bibitem{ceg2}
B.~Colbois, A.~El~Soufi, and A.~Girouard.
\newblock Isoperimetric control of the {S}teklov spectrum.
\newblock {\em J. Funct. Anal.}, 261(5):1384--1399, 2011.

\bibitem{cg1}
B.~Colbois and A.~Girouard.
\newblock The spectral gap of graphs and {S}teklov eigenvalues on surfaces.
\newblock {\em Electron. Res. Announc. Math. Sci.}, 21:19--27, 2014.

\bibitem{colbmat}
B.~Colbois and A.-M. Matei.
\newblock On the optimality of {J}. {C}heeger and {P}. {B}user inequalities.
\newblock {\em Differential Geom. Appl.}, 19(3):281--293, 2003.

\bibitem{dodziuk}
J.~Dodziuk.
\newblock Eigenvalues of the {L}aplacian on forms.
\newblock {\em Proc. Amer. Math. Soc.}, 85(3):437--443, 1982.

\bibitem{gpsurvey}
A.~{Girouard} and I.~{Polterovich}.
\newblock {Spectral geometry of the Steklov problem}.
\newblock {\em To appear in Journal of Spectral Theory}.
\newblock ar{X}iv:1103.2448.

\bibitem{kanai}
M.~Kanai.
\newblock Rough isometries, and combinatorial approximations of geometries of
  noncompact {R}iemannian manifolds.
\newblock {\em J. Math. Soc. Japan}, 37(3):391--413, 1985.

\bibitem{kanai1986}
M.~Kanai.
\newblock Rough isometries and the parabolicity of {R}iemannian manifolds.
\newblock {\em J. Math. Soc. Japan}, 38(2):227--238, 1986.

\bibitem{kokarev1}
G.~Kokarev.
\newblock Variational aspects of {L}aplace eigenvalues on {R}iemannian
  surfaces.
\newblock {\em Adv. Math.}, 258:191--239, 2014.

\bibitem{mantuano}
T.~Mantuano.
\newblock Discretization of compact {R}iemannian manifolds applied to the
  spectrum of {L}aplacian.
\newblock {\em Ann. Global Anal. Geom.}, 27(1):33--46, 2005.

\bibitem{mohar}
B.~Mohar.
\newblock The {L}aplacian spectrum of graphs.
\newblock In {\em Graph theory, combinatorics, and applications. {V}ol.\ 2
  ({K}alamazoo, {MI}, 1988)}, Wiley-Intersci. Publ., pages 871--898. Wiley, New
  York, 1991.

\bibitem{pinsker}
M.~S. Pinsker.
\newblock On the complexity of a concentrator.
\newblock {\em 7th International Teletraffic Conference}, pages 318/1--318/4,
  1973.

\bibitem{post}
O.~Post.
\newblock {\em Spectral analysis on graph-like spaces}, volume 2039 of {\em
  Lecture Notes in Mathematics}.
\newblock Springer, Heidelberg, 2012.

\end{thebibliography}

\end{document}